\documentclass[final,1p,times]{elsarticle}
\usepackage{fullpage}

\usepackage[T1]{fontenc}
\usepackage{amssymb}
\usepackage{amsthm}
\usepackage{bbm}
\usepackage{mathdots,mathtools}
\usepackage{mathbbol}
\usepackage{xcolor}
\usepackage[
]{hyperref}
\hypersetup{
	colorlinks,
	linkcolor={red!80!black},
	citecolor={blue!50!black},
	urlcolor={blue!80!black}
}
\theoremstyle{plain}
\newtheorem{theorem}{Theorem}[section]
\newtheorem{proposition}[theorem]{Proposition}
\newtheorem{lemma}[theorem]{Lemma}
\newtheorem{corollary}[theorem]{Corollary}

\newtheorem{assumcustomthm}{Assumption}
\newenvironment{customthm}[1]
{\renewcommand\theassumcustomthm{#1}\assumcustomthm}
{\endassumcustomthm}
\theoremstyle{remark}
\newtheorem{remark}{Remark}[section]
\newtheorem{example}{Example}

\journal{Stochastic Processes and their Applications}

\begin{document}

\begin{frontmatter}



\title{Uniform in time propagation of chaos for a Moran model}


\author[label1]{Bertrand Cloez}

\affiliation[label1]{organization={MISTEA, Univ Montpellier, INRAE, Institut Agro},
	city={Montpellier},
	postcode={34060}, 
	country={France}}

\author[label4,label2,label3]{Josu\'e Corujo\corref{cor1}}
\cortext[cor1]{Corresponding author.}

\affiliation[label4]{o={IRMA, Universit\'e de Strasbourg},
	c={Strasbourg},
	p={67084},
	cy={France}}
 
\affiliation[label2]{organization={CEREMADE, Universit\'e Paris-Dauphine, Universit\'e PSL, CNRS},
             city={Paris},
             postcode={75016},
             country={France}}

\affiliation[label3]{organization={Institut de Math\'ematiques de Toulouse, Universit\'e de Toulouse, Institut National des Sciences Appliqu\'ees},
             city={Toulouse},
             postcode={31077},
             country={France}}

\begin{abstract}
This article studies the limit of the empirical distribution induced by a mutation-selection multi-allelic Moran model. 
Our results include a uniform in time bound for the propagation of chaos in $\mathbb{L}^p$ of order $\sqrt{N}$, and the proof of the asymptotic normality with zero mean and explicit variance, when the number of individuals tend towards infinity, for the approximation error between the empirical distribution and its limit.
Additionally, we explore the interpretation of this Moran model as a particle process whose empirical probability measure approximates a quasi-stationary distribution, in the same spirit as the Fleming\,--\,Viot particle systems.
\end{abstract}




\end{frontmatter}

\section{Introduction and main results}

This paper is devoted to the study of a mutation-selection multi-allelic Moran model consisting on $N \in \mathbb{N}$ individuals, which can be of different allelic types belonging to a countable set $E$, equipped with the discrete metric.
The state space of the Moran model is the $N$ simplex
\[
\mathcal{E}_N := \left\{ \eta: E \rightarrow \mathbb{N} \mathrel{\Big|} \sum\limits_{x \in E} \eta(x) = N \right\}.
\]
The empirical distribution induced by $\eta \in \mathcal{E}_N$ is defined by
\[
m(\eta) = \sum_{x \in E} \frac{\eta(x)}{N} \delta_x \in \mathcal{M}_{1}(E),
\]
where $\mathcal{M}_{1}(E)$ is the set of probability measures on $E$. 
Let $Q$ be the generator of a continuous-time, non-explosive, irreducible Markov chain, and consider some rates $V_{\mu}(x,y) \ge 0$, for every $x \neq y \in E$ and $\mu \in \mathcal{M}_1(E)$.

The multi-allelic Moran model is a continuous-time Markov chain evolving on $\mathcal{E}_N$.
The process is at $\eta \in \mathcal{E}_N$ if there are $\eta(x)$ individuals of type $x$, for all $x \in E$.
Between reproduction events, the $N$ individuals evolve as independent copies of the mutation process generated by $Q = (Q_{x,y})_{x,y \in E}$.
In this sense we call $Q_{x,y}$, for $x,y \in E$, the \emph{mutation rates}.
Reproduction events consist of the death of an individual of type $x$, which is then removed from the population, and the reproduction of an individual of type $y$, which add an $y$ individual to the population. 
This happens at the rate $\eta(y)/N \cdot V_{m(\eta)}(x, y)$. 
Hence, the transition rate from $\eta \in \mathcal{E}_N$, with $\eta(x) > 0$, to $\eta - \mathbf{e}_x + \mathbf{e}_y$ is
\[
\eta(x) \left( Q_{x,y} + \frac{\eta(y)}{N} V_{m(\eta)}(x,y) \right),
\]
for every $x \neq y \in E$, where $\eta - \mathbf{e}_x + \mathbf{e}_y$ is the element in $\mathcal{E}_{N}$ satisfying
\[
\big(\eta - \mathbf{e}_x + \mathbf{e}_y\big)(z) =
\left\{
\begin{array}{ccl}
	\eta(z) & \text{ if } & z \notin \{x,y\},\\
	\eta(x)-1 & \text{ if } & z = x,\\
	\eta(y)+1 & \text{ if } & z = y.
\end{array}
\right.
\]
We will further detail particular examples but, for the moment, let us see that when $V_{m(\eta)}(x, y)$ is constant, each individual dies at the same rate and the parent is chosen uniformly at random over the individuals that are present in the population (this explains the term $\eta(y)/N$ in the transition rate). 
We can also interpret this rate by the opposed point of view: each individual reproduces at a constant rate, and the dying individual is chosen uniformly at random. 
This is often called neutral selection in ecology literature, but our models allow choosing various non constant $V_{m(\eta)}(x, y)$.
In this sense, we call the rates $V_\mu(x,y)$, for $x,y \in E$ and $\mu \in \mathcal{M}_1(E)$, the \emph{selection rates}.

Note that the reproduction dynamics depends in general on both the types of the parent and the offspring, and may also depend on the empirical distribution induced by the configuration of the population at the current time, in a sense that we will clarify further along in Assumptions \ref{assump:gral_slection rate} and \ref{assump:additive+symmetric}.
The generator of the Moran model is denoted $\mathcal{Q} := \mathcal{Q}^{\mathrm{mut}} + \mathcal{Q}^{\mathrm{sel}}$, where $\mathcal{Q}^{\mathrm{mut}}$ and $\mathcal{Q}^{\mathrm{sel}}$ act on every function $f \in \mathcal{B}_b(\mathcal{E}_{N})$ as follows
\begin{align*}
	(\mathcal{Q}^{\mathrm{mut}} f)(\eta) &= \sum_{x, y \in E} \eta(x) Q_{x,y} [f(\eta - \mathbf{e}_x + \mathbf{e}_y) - f(\eta)],\\
	(\mathcal{Q}^{\mathrm{sel}} f)(\eta) &= \frac{1}{N} \sum_{x, y  \in E} \eta(x) \eta(y) V_{m(\eta)}(x,y) [f(\eta - \mathbf{e}_x + \mathbf{e}_y) - f(\eta)],
\end{align*}
for every $\eta \in \mathcal{E}_N$.
Throughout the paper, the following boundedness condition holds:
\begin{equation}\label{eq:bound_condition}
	\|V\| := \sup_{\mu \in \mathcal{M}_1(E)} \sup_{x,y \in E} V_\mu(x,y) < \infty.
\end{equation}
Note that the non-explosion of the process generated by $Q$ and the bound condition \eqref{eq:bound_condition} let out the possibility of an infinite number of jumps in finite time.
Thus, the process generated by $\mathcal{Q}$ is well-defined for all $t \ge 0$.

This Moran model is an extension, for $K > 2$, of the model studied by \cite{cordero_deterministic_2017}.
In general, when generalising the Moran model for more than two allelic types, the selection rates are taken depending only on the children type, i.e.\ $V_{\mu}(x,y) = V^{\mathrm{b}}(y)$, for all $x,y \in E$ and $\mu \in \mathcal{M}_1(E)$, which is called \emph{selection at birth} or \emph{fecundity selection} \cite{durrett_probability_2008,muirhead_modeling_2009,etheridge_mathematical_2011}.
Moreover, in biological applications it has been also considered models with \emph{selection at death} or \emph{viability selection}, when the selection rates only depend on the parent type, i.e.\ $V_\mu(x,y) = V^{\mathrm{d}}(x)$ \cite{muirhead_modeling_2009}, for all $x,y \in E$ and $\mu \in \mathcal{M}_1(E)$.
However, the importance of this last model is beyond its biological interpretations: this process is also called \emph{Fleming\,--\,Viot particle process}, which is an interacting particle process 
intended for the approximation of a \emph{quasi-stationary distribution} (QSD) of an absorbing Markov chain conditioned on non-absorption.
These particle processes have attracted lots of attention in recent years.
See for instance
\cite{MR1956078,villemonais_general_2014,Cerou2020,zbMATH07298000} for general state spaces,
\cite{ferrari_quasi_2007,MR3156964,Galton-Watson_FV_2016,cloez_quantitative_2016} for countable state spaces, and even 
\cite{asselah_quasistationary_2011,Lelievre2018} for finite state spaces.

\subsection*{Structure of the paper} 
The rest of the paper is organised as follows.
In Section \ref{sec:main_results} we state the main results of the paper and comment some of their consequences.
In Section \ref{sec:examples} we consider several examples of mutation and selection rates satisfying Assumptions \ref{assump:gral_slection rate} and \ref{assump:additive+symmetric}, and such that the exponential uniform in time convergence of the normalised semigroups, namely Assumption \ref{assump:ergod_normalised}, also hold.
We end this section with a discussion about some of the possible extensions.
Finally, in Section \ref{sec:proofs}, we prove the main results, and other results and proofs are deferred to the appendices.

\subsection{Main results}\label{sec:main_results}

Let us get some insight into the limit of the empirical measure induced by this particle process when the number of particles tends towards infinity.
Let us denote by $(\eta_t)_{t \ge 0}$ the continuous-time Markov chain on $\mathcal{E}_N$, generated by $\mathcal{Q}$.
Although the process generated by $\mathcal{Q}$ clearly depends on $N$ and a better notation would be $(\eta_t^{(N)})_{t \ge 0}$, we keep this dependence implicit for the sake of simplicity.
By the Kolmogorov equation we know that
\(
\partial_t \mathbb{E}_{\eta} [m_x(\eta_t)] = \mathbb{E}_{\eta} \left[ (\mathcal{Q} \, m_x)(\eta_t) \right],
\)
where $m_x$ stands for the empirical distribution induced by $\eta$ on the point $x \in E$, i.e.\
\(
m_x : \eta \mapsto {\eta(x)}/{N}.
\)
Let us thus compute $\mathcal{Q} \, m_x$. 
On the one hand, it is easy to get
\begin{equation}\label{eq:gen_mut}
	\big( \mathcal{Q}^{\mathrm{mut}} m_x \big) (\eta) = \sum_{y \in E} Q_{x,y} m_{y}(\eta),	
\end{equation}
for every $x \in E$, for all $\eta \in \mathcal{E}_N$.
On the other hand,
\begin{align}
	\big( \mathcal{Q}^{\mathrm{sel}} m_x \big) (\eta) 
	&= -m_x(\eta) \sum_{y \in E} m_{y}(\eta) [V_{m(\eta)}(x,y) - V_{m(\eta)}(y,x)]. \label{eq:gen_sel}
\end{align}
Finally, we get
\[
\partial_t \mathbb{E}_{\eta} [m_x(\eta_t)] = \sum_{y \in E} Q_{x,y} \mathbb{E}_{\eta}[m_y(\eta_t)] - \sum_{y \in E} [V_{m(\eta_t)}(x,y) - V_{m(\eta_t)}(y,x)] \mathbb{E}_{\eta}[m_x(\eta_t) m_y(\eta_t)].
\]
When the number of individuals $N$ is large, we expect the Moran process to exhibit a \emph{propagation of chaos phenomenon} and thus the empirical distribution induced by the process approximates the solution of the following nonlinear system of ordinary differential equations:
\begin{equation}\label{eq:EDO_without_symmetric}
    \partial_t \gamma_t (x) = \sum_{y \in E} Q_{x,y} \gamma_t(y) - \sum_{y \in E} [V_{\gamma_t}(x,y) - V_{\gamma_t}(y,x)] \gamma_t(x) \gamma_t(y),
\end{equation}
for all $x \in E$.
For every function $\phi$ on $E$ we thus get the nonlinear differential equation
\begin{align}
	\partial_t \gamma_t(\phi) 
	&= \gamma_t( Q_{\gamma_t} \phi ), \label{eq:preODE}
\end{align}
where $Q_{\gamma} := Q + \Pi_\gamma$ and
\begin{equation*}
	\Pi_\gamma \phi: x \mapsto \sum_{y \in E} \gamma(y) V_{\gamma}(x,y) [\phi(y) - \phi(x)],
\end{equation*}
for every probability distribution $\gamma$ on $E$.

The main results we provide in this article are related to the speed of convergence of $\big(m(\eta_t) \big)_{t \ge 0}$ towards $\big( \gamma_t \big)_{t \ge 0}$ when $N \to \infty$.

\subsubsection{Propagation of chaos with general selection rate}

Let $\mathcal{B}_b(E)$ be the set of bounded functions on $E$ for the uniform norm, which is denoted by $\| \cdot \|$ and defined by
\(
\|\phi\| := \sup_{x \in E} |\phi(x)|.
\) 
Let us also denote $\mathcal{B}_1(E) := \{\phi : E \to \mathbb{R}: \; \|\phi\| \le 1\}$.
For two probability distributions  $\mu_1, \mu_2 \in \mathcal{M}_1(E)$ the total variation distance is defined as follows:
\[ 
\| \mu_1 - \mu_2 \|_{\mathrm{TV}} := \sup_{A \subset E} |\mu_1(A) - \mu_2(A)| = \frac{1}{2} \sup_{\phi \in \mathcal{B}_1(E) } \left| \mu_1 (\phi) - \mu_2 (\phi) \right| = \frac{1}{2} \sum_{x \in E} |\mu_1(x) - \mu_2(x)|,
\]	
where $\mu(\phi)$ stands for the mean of $\phi$ with respect to $\mu \in \mathcal{M}_1(E)$.

First, we consider the case where the selection rates satisfy the following hypothesis.
\begin{customthm}{$(\mathrm{G1})$}[General selection rate] \label{assump:gral_slection rate}
	There exist $V^{\mathrm{d}}_i, V^{\mathrm{b}}_i$, for $i \ge 1$, and
	a continuous, nonnegative function $\mu \mapsto V_\mu^{\mathrm{s}}$ from
	$(\mathcal{M}_1(E), \|\cdot\|_{\mathrm{TV}})$ to $(\mathcal{B}_b(E \times E), \|\cdot\|)$
	such that $V_\mu^{\mathrm{s}}$ is symmetric in $E \times E$ for every $\mu$, and
	\begin{equation}\label{eq:Vdecomp}
	    V_\mu(x,y) = \sum_{i \ge 1} V^{\mathrm{d}}_i(x) V^{\mathrm{b}}_i(y) + V_\mu^\mathrm{s}(x,y), 
	\end{equation}
	for all $\mu \in \mathcal{M}_1(E)$ and $x,y \in E$.
	In addition, $	V_\mu - V^\mathrm{s}_\mu \in \mathcal{B}_b(E \times E)$ and 
	\begin{equation}\label{eq:boundednes_condition}
	    \left\{
	\sum_{i \ge 1} |V^{\mathrm{d}}_i - V^{\mathrm{b}}_i|,\;
	\sup_{i \ge 1} V^{\mathrm{d}}_i,\;
	\sup_{i \ge 1} V^{\mathrm{d}}_i
	\right\} \subset \mathcal{B}_b(E).
	\end{equation}
\end{customthm}

\begin{remark}[Decomposition of $V \in \mathcal{B}_b(E \times E)$]
    Since the state space $E$ is countable, any function $V \in \mathcal{B}_b(E \times E)$ can be written as follows
    \[
        V(x,y) = \sum_{z \in E} \pmb{1}_{ \{z\}} (x) V(z,y) 
        = \sum_{z \in E} \pmb{1}_{ \{z\}} (y) V(x,z).
    \]
    Thus, it is always possible to decompose $V$ in a countable sum as in \eqref{eq:Vdecomp}. 
    In particular, Assumption \ref{assump:gral_slection rate} is always satisfied when $E$ is finite. 
    However, when $E$ is infinite such decomposition does not necessarily satisfy the boundedness conditions in Assumption \ref{assump:gral_slection rate}, namely it does not necessarily satisfy
    \[
	\sum_{z \in E} |V(z, \cdot)| \in \mathcal{B}_b(E)	\text{ or } 
	\sum_{z \in E} |V(\cdot, z)| \in \mathcal{B}_b(E).
    \]
    Consider for example the case where $E = \mathbb{N}$ and $V_\mu(x,y) = a_x b_y$, where $(a_x)_{x \in \mathbb{N}}$ and $(b_y)_{x \in \mathbb{N}}$ are bounded and are the general terms of two divergent series.
    Taking $V^\mathrm{s} \equiv 0$,
    $V_z^\mathrm{d}: x \mapsto \pmb{1}_{\{z\}}(x)$ and $V_z^\mathrm{b} : y \mapsto a_z b_y$ (or $V_z^\mathrm{d}: x \mapsto a_x b_z$ and $V_z^\mathrm{b} : y \mapsto \pmb{1}_{\{z\}}(y)$), then condition \eqref{eq:boundednes_condition} is not satisfied.
    However, taking $V_1^\mathrm{d} : x \mapsto a_x$, $V_1^\mathrm{b} : y \mapsto b_y$ and $V_i^\mathrm{d} V_i^\mathrm{b} \equiv 0$, for all $i \ge 2$, one easily checks that Assumption \ref{assump:gral_slection rate} is satisfied.
\end{remark}

\begin{remark}[Rates giving the same mean-field limit]
    Note that, because of \eqref{eq:EDO_without_symmetric}, the mean-field limit is invariant as soon as the same the terms $V_\mu(x,y) - V_{\mu}(y,x)$, for all $x,y \in E$, do not change.
    In particular, the mean-field limit does not depend on the symmetric term in Assumption \ref{assump:additive+symmetric}.
\end{remark}

Now, for every $V \in \mathcal{B}_b(E \times E)$ and $\gamma_0 \in \mathcal{M}_1(E)$, let us define the ``implicit Feynman\,--\,Kac'' flow: \begin{equation}\label{eq:F-Kimplicit}
    \gamma_t^V(\phi) := \mathbb{E}_{\gamma_0}\left[ \phi(X_t) \exp\left\{ \int_0^t \sum_{x \in E} [V(X_s, x) - V(x, X_s)] \gamma_s^V(x) \mathrm{d}s \right\} \right].
\end{equation}
Then, $(\gamma_t^V)_{t \ge 0}$ is a solution of the Cauchy problem
\begin{equation}\label{eq:postODE}
      \partial_t \mu_t(\phi) 
    = \mu_t\big( \widetilde{Q}_{\mu_t} \phi \big), \text{ with } \mu_0(\phi) = \gamma_0(\phi),
\end{equation}
where
\[
    \widetilde{Q}_{\mu} \phi: x \mapsto (Q \phi)(x) + \sum_{y \in E} \mu(y) V(x,y) [\phi(y)-\phi(x)].
\] 
See \cite[p.\ 25]{DelMoral2004} and the references therein for more details on the implicit Feynman\,--\,Kac.
Besides, as suggested in the above-mentioned reference, following the martingale arguments in \cite{del_moral_moran_2000}, one can check that $(\gamma_t^V)_{t \ge 0}$ is in fact the unique solution of \eqref{eq:postODE}.
According to these references, we claim the following result.
\begin{lemma}[Existence of the solution of \eqref{eq:postODE}] \label{assump:gral_slection rate_existence}
	Assume that Assumption \ref{assump:gral_slection rate} is satisfied.
	For every $\mu_0 \in \mathcal{M}_1(E)$, there is a unique solution $(\mu_t)_{t \ge 0}$ of the differential equation \eqref{eq:postODE}, and it is given by \eqref{eq:F-Kimplicit} taking $V = V_\mu - V_\mu^{\mathrm{s}}$.
\end{lemma}

Consider also the following assumption.
\begin{customthm}{$(\mathrm{I})$}[Initial condition]\label{assump:initial_condition}
	The empirical measure induced by the particle process at $t = 0$ converges towards the initial distribution $\mu_0 \in \mathcal{M}_1(E)$ in $\mathbb{L}^p$, for every $p \ge 1$.
	More precisely, for every $p \ge 1$, there exists a constant $C_p > 0$ such that
	\[
	\sup_{ \phi \in \mathcal{B}_1(E) } \mathbb{E}[|m(\eta_0)(\phi) - \mu_0(\phi)|^p] \le \frac{C_p}{N^{p/2}}.
	\]
\end{customthm}

Note that Assumption \ref{assump:initial_condition} is verified when initially all the particles are sampled independently with distribution $\mu_0$, as the next lemma shows.
\begin{lemma}[Control of the initial error]\label{lemma:control_initial_condition}
	Assume that initially the $N$ particles are sampled independently according to $\mu_0 \in \mathcal{M}_1(E)$. 
	Then, Assumption \ref{assump:initial_condition} is verified.
\end{lemma}
The proof of Lemma  \ref{lemma:control_initial_condition} is deferred to Appendix \ref{sec:appendix}.
We include Assumption \ref{assump:initial_condition} in order to be able to apply our results to a wider class of situations, than that described in Lemma \ref{lemma:control_initial_condition}.

\begin{theorem}[Propagation of chaos]\label{thm:propagation_chaos}
	Suppose that Assumptions \ref{assump:gral_slection rate}, and \ref{assump:initial_condition} are verified. 
	Then, for every $T \ge 0$ and $p \ge 1$, there exist two positive constants $\alpha_{p}$ and $\beta_p$, possibly depending on $p$,  such that
	\[
	\sup_{ \phi \in \mathcal{B}_1(E) }  \mathbb{E} \left[  \sup_{t \in [0,T]} | m(\eta_t)(\phi) - \mu_t(\phi) |^p \right]^{1/p} \le
	\alpha_p \frac{\sqrt{1 + T}}{\sqrt{N}} \mathrm{e}^{\beta_p T},
	\]
	where $(\mu_t)_{t \ge 0}$ is as in Lemma \ref{assump:gral_slection rate_existence}, with initial condition $\mu_0 \in \mathcal{M}_1(E)$ as in Assumption \ref{assump:initial_condition}.
\end{theorem}
The proof of Theorem \ref{thm:propagation_chaos} is deferred to Section \ref{sec:proofThm1}.

Let $(x_n)_{n \ge 1}$ be an enumeration of the elements in $E$.
We define the following distance in $\mathcal{M}_1(E)$:
\begin{equation*}
	\|\mu_1 - \mu_2 \|_{\mathrm{w}} := \sum_{k \ge 1} 2^{-k} |\mu_1(x_k) - \mu_2(x_k)|.
\end{equation*}
Note that the space $\mathcal{M}_1(E)$ with the convergence in law (the weak topology) is metrisable with this distance.
As a consequence of Theorem \ref{thm:propagation_chaos} we get the following result.
\begin{corollary}[Convergence of the empirical measure]\label{cor:convergence_in_normL}
	Suppose that Assumptions \ref{assump:gral_slection rate}, and \ref{assump:initial_condition} are verified. 
	Then, for every $T \ge 0$ and $p \ge 1$, there exist two positive constants $\alpha_{p}$ and $\beta_p$, such that
	\[
	\mathbb{E}\left[ \left( \sup_{t \in [0,T]} \| m(\eta_t) - \mu_t \|_\mathrm{w} \right)^p \right]^{1/p} \le 
	\alpha_p \frac{\sqrt{1 + T}}{\sqrt{N}} \mathrm{e}^{\beta_p T},	
	\]
	where $(\mu_t)_{t \ge 0}$ is as in Lemma \ref{assump:gral_slection rate_existence}, with initial condition $\mu_0 \in \mathcal{M}_1(E)$ as in Assumption \ref{assump:initial_condition}.
\end{corollary}
Corollary \ref{cor:convergence_in_normL} is proved in Section \ref{sec:proofThm1}.

Note that this ensures a functional convergence in $\mathbb{L}^p \big(\mathcal{C}([0,T], \mathcal{M}_1(E)) \big)$: 
\[
m(\eta_\cdot) \xrightarrow[N \rightarrow \infty]{\mathbb{L}^p} \mu_\cdot,
\]
with an estimation of the speed of convergence.
Furthermore, Theorem  \ref{thm:propagation_chaos}, for $p=4$, and a Borel\,--\,Cantelli argument imply the convergence $m(\eta_\cdot) \xrightarrow{\mathrm{c.c.}} \mu_\cdot$ in the weak sense, where $\mathrm{c.c.}$ denotes the complete (or universal) convergence (cf.\ \cite[Def.\ 1.6]{2013Gut}).
In particular, this implies
\(
m(\eta_\cdot) \xrightarrow{ \mathrm{a.s.} } \mu_\cdot,
\)
when $N \rightarrow \infty$, in the weak sense, no matter in which space the random variables are coupled. 

Theorem \ref{thm:propagation_chaos} is a generalisation for multi-allelic Moran models with more than two allelic types, of Proposition 3.1 in \cite{cordero_deterministic_2017}, where the uniform convergence on compacts time intervals in probability is proved.	
The speed of convergence in Theorem \ref{thm:propagation_chaos} can also be related to existing results that ensure the convergence of the empirical measure induced by a Moran type (or Fleming\,--\,Viot) particle process towards the law of an absorbing process conditioned to non-absorption.
See for instance \cite[Prop.\ 3.5]{DelMoralMiclo2000} \cite[Lemma 3.1]{DelMoral2011}, \cite[Thm.\ 2.2]{villemonais_general_2014}, 
\cite[Thm.\ 1.1]{cloez_fleming-viot_2016},
and \cite[Thm.\ 5.10 and Cor.\ 5.12]{MR4193898}.
See also \cite[Thm.\ 3.1 and Rmk.\ 3.2]{2015Benaim&Cloez} where the almost sure convergence  (and also the complete convergence) is proved when the state space is finite.
As far as we know, Theorem \ref{thm:propagation_chaos} and Corollary \ref{cor:convergence_in_normL} are the first results ensuring the convergence uniformly on compacts in $\mathbb{L}^p$, for all $p \ge 1$, with speed of convergence of order $1/\sqrt{N}$ for multi-allelic Moran models
with general selection rates in the sense of Assumption \ref{assump:gral_slection rate},
in discrete countable state spaces, not necessarily finite.
The idea behind the proof is closed to the methods in \cite{MR2262944}: it consists in finding a martingale indexed by the interval $[0, T]$, whose terminal value at time $T$ is precisely
$m(\eta_T)(\phi) - \mu_T(\phi)$ plus a term whose $\mathbb{L}^p$ norm can be controlled, for any $\phi \in \mathcal{B}_b(E)$.
Thereafter, the final result comes by a  Gr\"{o}nwall type argument, similarly to the proof of Proposition 1 in \cite{SalezMerle2019}.
Nevertheless, \cite{MR2262944} does not contain any uniform bound as in Theorem \ref{thm:propagation_chaos}.

\subsubsection{Uniform in time propagation of chaos for additive selection rates}

Under a more specific expression for the selection rates, we prove a uniform in time bound for the convergence of $\big(m(\eta_t) \big)_{t \ge 0}$ towards $\big( \mu_t \big)_{t \ge 0}$, when $N \rightarrow \infty$.
Consider the following kind of selection rates that we call \emph{additive selection}.
\begin{customthm}{$(\mathrm{C}1)$}[Additive selection] \label{assump:additive+symmetric}
	The selection rates are uniformly bounded as in \eqref{eq:bound_condition}.
	Moreover, there exist two continuous nonnegative functions 
	\(
	\mu \mapsto V^{\mathrm{d}}_{\mu}
	\)
	and
	\(
	\mu \mapsto V^{\mathrm{b}}_{\mu},
	\)
	from $(\mathcal{M}_1(E), \|\cdot\|_{\mathrm{TV}})$ to $(\mathcal{B}_b(E), \|\cdot\|)$;
	and a continuous, nonnegative function $\mu \mapsto V_\mu^{\mathrm{s}}$ from
	$(\mathcal{M}_1(E), \|\cdot\|_{\mathrm{TV}})$ to $(\mathcal{B}_b(E \times E), \|\cdot\|)$
	such that $V_\mu^\mathrm{s}$ is symmetric on $E \times E$, for every $\mu \in \mathcal{M}_1(E)$ and
	\begin{equation*}
		V_\mu(x,y) = V_\mu^{\mathrm{d}}(x) + V_\mu^{\mathrm{b}}(y) + V_\mu^\mathrm{s}(x,y), 
	\end{equation*}
	for all $x,y \in E$ and $\mu \in \mathcal{M}_1(E)$.
	Furthermore, there exists a function $\Lambda \in \mathcal{B}_b(E)$ such that 
	\begin{equation}\label{eq:defW}
		\Lambda(x) = V_\mu^{\mathrm{b}}(x) - V_\mu^{\mathrm{d}}(x),
	\end{equation}
	for every $x \in E$.
\end{customthm}

\begin{remark}[Selection rates independent on $\mu$]
	When the selection rates do not depend on $\mu$, Assumption \ref{assump:additive+symmetric} reduces to the existence of $V^\mathrm{d}, V^\mathrm{b} \in \mathcal{B}_b(E)$ and a symmetric  $V^\mathrm{s} \in \mathcal{B}_b(E \times E)$ such that
	\[
	V(x,y) = V^{\mathrm{d}}(x) + V^{\mathrm{b}}(y) + V^\mathrm{s}(x,y).	
	\]
	Let $\Lambda \in \mathcal{B}_b(E)$ be a fixed function.
	Typical examples of functions $V^\mathrm{b}$ and $V^\mathrm{d}$ satisfying this condition are
	\[
	V^\mathrm{b} = (\Lambda -c)^+ \text{ and } V^\mathrm{d} = (\Lambda -c)^-,
	\]
	for a fixed constant $c \in \mathbb{R}$, 
	where we use the standard notation 
	\(
	(x)^+ := \max\{x,0\} \text{ and } (x)^- := -\min\{x,0\}.
	\)
	These are in fact the selection rates considered by Angeli et al.\ \cite[\S\ 3.3]{Angeli2021} in the context of  cloning algorithms.
	Moreover, the case $c = 0$ is considered in Example 3.1-(2) in \cite{MR2262944}.
	Note that in this case, Assumption \ref{assump:gral_slection rate} is also verified.
	
	From a biological point of view, the parameter $c \in \mathbb{R}$ can be seen as a fitness parameter.
	Let us assume that $V_\mu^\mathrm{s}$ is null for simplicity, and denote by $\xi_t^{(i)}$ the type of the $i$-th individual, for $1 \le i \le N$, at time $t \ge 0$.
	Then, if $\Lambda(\xi_t^{(i)}) \le c$, the $i$-th individual dies and another randomly chosen individual reproduces with rate $(\Lambda(\xi_t^{(i)}) - c)^-$. 
	Otherwise, if $\Lambda(\xi_t^{(i)}) \ge c$ a random chosen individual dies and the $i$-th individual  reproduces with rate $(\Lambda(\xi_t^{(i)}) - c)^+$.
	
	Another example of particular interest is when 
	\(
	V^\mathrm{b} = 0.
	\)
	Notice that the Moran process with these selection rates is in fact a Fleming\,--\,Viot particle process (cf.\ \cite{ferrari_quasi_2007}).

\end{remark}

\begin{remark}[Selection rates depending on $\mu$]
	Consider a fixed function $\Lambda \in \mathcal{B}_b(E)$.
	Typical examples of functions $V_\mu^\mathrm{b}$ and $V_\mu^\mathrm{d}$ are:
	\[
	V_\mu^\mathrm{b} = \big(\Lambda - \mu(\Lambda)\big)^+ \text{ and } V_\mu^\mathrm{d} = \big(\Lambda - \mu(\Lambda)\big)^-.
	\]
	These are the selection rates considered in \cite[\S\ 1.5.2, p.\ 35]{DelMoral2004}, see also Example 3.1-(3) in \cite{MR2262944}.
	In this case, the biological interpretation of $\mu(\Lambda)$ is similar to that of the parameter $c$ in the previous remark.
	Indeed, the fitness coefficient evolves in time according to the evolution of the population.
	
\end{remark}

Consider that Assumption \ref{assump:additive+symmetric} is satisfied.
Then, from \eqref{eq:preODE} we can recover the nonlinear differential equation
\begin{equation}\label{eq:fokker-plack}
	\partial_t \gamma_t(\phi) = \gamma_t \big((Q + \Lambda) \phi - \gamma_t(\Lambda) \phi \big),
\end{equation}
where $\Lambda$ is as defined in \eqref{eq:defW}.
Consider the Feynman\,--\,Kac semigroup $(P_t^\Lambda)_{t \ge 0}$, where
\begin{equation*}
	P_t^\Lambda (\phi) : x \mapsto \mathbb{E}_x\left[ \phi(X_t) \exp\left\{\int_0^t \Lambda(X_s) \mathrm{d}s \right\} \right]
\end{equation*}
whose generator is $Q + \Lambda$. 
Let us define the normalised version of this semigroup as follows
\begin{equation}\label{def:eq_normalised_semigroup}
		\mu_t (\phi) := \frac{\mu_0 P_t^\Lambda (\phi)}{\mu_0 P_t^\Lambda (\pmb{1})},
	\end{equation}
	where $\pmb{1}$ denotes the all-one function on $E$.
	Then, $(\mu_t)_{t \ge 0}$ is the solution of the nonlinear differential equation \eqref{eq:fokker-plack} with initial value $\mu_0(\phi)$ for $t = 0$ \cite[Eq.\ (1.17)]{DelMoral2004}.
	
	Furthermore, the definition of $(\mu_t)$ in \ref{def:eq_normalised_semigroup} is invariant by translation of the function $\Lambda$, i.e.\ for every real $\beta$ we have that $\mu_t (\phi) = {\mu_0 P_t^{\Lambda - \beta} (\phi)}/{\mu_0 P_t^{\Lambda - \beta} (\pmb{1})}$.
	In particular, one can always interpret $(\mu_t)_{t \ge 0}$ as the distribution of an absorbed Markov chain conditioned to non-absorption up to time $t$ with killing rate $\kappa = \sup \Lambda - \Lambda$.
	This naturally relates the study of the behaviour of $(\mu_t)_{t \ge 0}$ when $t \rightarrow \infty$, to the theory of quasi-stationary distributions (QSD), see e.g.\ \cite[\S\ 7]{zbMATH06190205} and \cite[\S\ 1.2]{Corujo_thesis}.

Consider the following assumptions related to the control in the norm $\mathbb{L}^p$ of the initial error and the exponential convergence of $(\mu_t)_{t \ge 0}$, as defined by \eqref{def:eq_normalised_semigroup}, towards a unique limit, for every initial distribution on $\mu_0 \in \mathcal{M}_1(E)$.

\begin{customthm}{$(\mathrm{C}2)$}[Uniform exponential ergodicity of the normalised semigroup]\label{assump:ergod_normalised}
	There exist a distribution $\mu_{\infty}\in \mathcal{M}_1(E)$ and $C, \gamma > 0$, such that for every initial distribution $\mu_0 \in \mathcal{M}_1(E)$ and for all $t \ge 0$:
	\begin{equation}
		\|\mu_{t}  - \mu_{\infty}\|_{\mathrm{TV}} \le C \mathrm{e}^{-\gamma t}, \label{eq:assump_ergo_normalised} 
	\end{equation}
	where $(\mu_t)_{t \ge 0}$ is defined as in \eqref{def:eq_normalised_semigroup}.
\end{customthm}

Assumption \ref{assump:ergod_normalised} is always satisfied when $E$ is finite. 
In this case, inequality \eqref{eq:assump_ergo_normalised} was proved by Darroch and Seneta \cite{Darroch_Seneta1967} and the result comes as a consequence of the Perron\,--\,Frobenius Theorem (see \cite[Thm.\ 8]{meleard_quasi-stationary_2012} for the specific context of quasi-stationary distributions).
The case where $E$ is countable is more delicate and has attracted lots of attention and several methods have been applied.
See for example the work of Del Moral and Miclo \cite{AFST_2002_6_11_2_135_0}, the articles of Champagnat and  Villemonais on exponential convergence towards the quasi-stationary distribution, specifically \cite{zbMATH06540706,Villemonais2017}, and also the works of Bansaye et al.\ \cite{Bansaye2019,zbMATH07181472}.
See also the recent work of Del Moral et al.\ \cite{DelMoral2022} on the stability of positive semigroups.
In Section \ref{sec:examples}, we provide several examples of process where Assumption \ref{assump:ergod_normalised} holds.

We are now in a position to state our main results for the multi-allelic Moran model with additive selection.

\begin{theorem}[Uniform in time propagation of chaos]\label{thm:uniformLp}
	Under Assumptions \ref{assump:initial_condition}, \ref{assump:additive+symmetric} and \ref{assump:ergod_normalised}, for every $p \ge 1$, there exists a constant $C_p$, such that
	\[
	\sup_{ \phi \in \mathcal{B}_1(E) } \sup_{t \ge 0} \mathbb{E} \big[ | m(\eta_t)(\phi) - \mu_t(\phi) |^p \big]^{1/p} \le \frac{C_p}{\sqrt{N}}.
	\] 
\end{theorem}
The proof of Theorem \ref{thm:uniformLp} is based on the methods developed by Rousset \cite{MR2262944} and it is deferred to Appendix \ref{sec:proof_thm2-3}.

\begin{corollary}[Convergence of the empirical measure]\label{cor:convergence_in_normL_uniform}
	Suppose that Assumptions \ref{assump:initial_condition}, \ref{assump:additive+symmetric} and \ref{assump:ergod_normalised} are verified. 
	Then, for every $p \ge 1$, there exists a constant $C_{p} > 0$, such that
	\[
	\sup_{t \ge 0 } \mathbb{E}\big[ \big(  \| m(\eta_t) - \mu_t \|_\mathrm{w} \big)^p \big]^{1/p} \le \frac{C_{p}}{\sqrt{N}}.
	\]
\end{corollary}
The proof of Corollary \ref{cor:convergence_in_normL_uniform} is analogous to that of Corollary \ref{cor:convergence_in_normL}, and we skip it for the seek of brevity.

\begin{remark}[Almost sure convergence]\label{rmk:almostsure_conv}
	Corollary \ref{cor:convergence_in_normL_uniform}, for $p=4$, and a Borel\,--\,Cantelli argument imply the convergence $m(\eta_T) \xrightarrow{ \mathrm{c.c.} } \mu_T$ in $\mathcal{M}_1(E)$, when $N \rightarrow \infty$, for every $T \ge 0$, where $\mathrm{c.c.}$ denotes the complete (or universal) convergence.
	In particular, this implies
	\(
	m(\eta_T) \xrightarrow{ \mathrm{a.s.} } \mu_T,
	\)
	when $N \rightarrow \infty$, for every $T \ge 0$.
	Note that in contrast with Corollary \ref{cor:convergence_in_normL}, Corollary \ref{cor:convergence_in_normL_uniform} does not ensure the convergence in $\mathcal{C}([0,T], \mathcal{M}_1(E))$.
\end{remark}

As a consequence of Theorem \ref{thm:uniformLp}, it is also possible to control the bias of one particle, following a similar method of that in \cite[Thm.\ 4.2]{MR2262944}.
Thus, since these results are somehow expected, we have decided, for the sake of brevity, to place the statement (Theorem \ref{thm:TVbound}) and its proof in the appendices.

The following result ensures the exponential ergodicity of the unnormalised semigroup.

\begin{lemma}[Exponential ergodicity of the unnormalised semigroup]\label{corollary:exp_ergodicity_nonnormalised}
	Suppose that Assumptions \ref{assump:additive+symmetric} and \ref{assump:ergod_normalised} are verified.
	Then, there exists a unique triplet $(\mu_{\infty}, h, \lambda) \in \mathcal{M}_1(E) \times \mathcal{B}_b(E) \times \mathbb{R}$, of eigenelements of $Q+ \Lambda$ such that $h$ is strictly positive, $\mu_{\infty}(h) = 1$,
	\[
	\mu_{\infty} P_t^\Lambda = \mathrm{e}^{\lambda t} \mu_{\infty} \text{ and } P_t^\Lambda (h) = \mathrm{e}^{\lambda t}h.
	\] 
	Moreover, there exist  $C, \gamma > 0$ such that for all $t \ge 0$:
	\begin{equation}
		\sup_{ \mu_0 \in \mathcal{M}(E)}\| \mathrm{e}^{-\lambda t} \mu_0 P_{t}^{\Lambda} - \mu_0(h) \mu_\infty \|_{\mathrm{TV}} \le C \mathrm{e}^{- \gamma t}. \label{eq:assump_ergo_FeynmanKac}
	\end{equation}
	Furthermore, $\lambda \le 0$ whether $\Lambda \le 0$.
\end{lemma}

This result is basically a consequence of Theorem 2.1 of \cite{Villemonais2017}.
Lemma \ref{corollary:exp_ergodicity_nonnormalised} establishes an exponential control on the speed of convergence of the unnormalised semigroup.
A similar estimate is stated by Angeli et al.\ \cite[Assumption 2.2]{Angeli2021} as hypothesis.
However, their assumption implies that the eigenfunction $h$ is constant, which in practice makes their assumption only valid when $\Lambda$ ($\mathcal{V}$ in their notation) is constant.

Let us define
\begin{equation}
	S_\mu(\phi) := \sum_{x,y \in E} (\phi(x) - \phi(y))^2 V_\mu^\mathrm{s}(x,y) \mu(x) \mu(y), \label{eq:defSmu}
\end{equation}
for every $\phi \in \mathcal{B}_b(E)$, and the operator $W_{t,T}$ for $t \le T$ as follows
\begin{equation}
	W_{t,T}	: \phi \mapsto \frac{P_{T-t}^\Lambda (\phi)}{ \mu_t \left( P_{T-t}^\Lambda(\pmb{1}) \right)}, \label{eq:defWtT_intro}
\end{equation}

Our last two results are addressed  to the study of the asymptotic square error of the approximation of $\mu_{T}$ by $m(\eta_T)$ when $T,N \rightarrow \infty$.
These results are particularly important when the Moran process is used for approximating a quasi-stationary distribution.
Let us define the asymptotic quadratic errors:
\begin{align*}
	\sigma^2_T(\phi) &:= \lim\limits_{N \rightarrow \infty}  N \mathbb{E} \left[ \big( m(\eta_T)(\phi) - \mu_{T}(\phi) \big)^2 \right], \nonumber \\
	\sigma^2_\infty(\phi) &:= \lim\limits_{T \rightarrow \infty} \sigma_T^2(\phi),
\end{align*}
for every $\phi \in \mathcal{B}_b(E)$.
First, we prove the asymptotic normality of the bias and we provide explicit expressions for $\sigma^2_T(\phi)$ and $\sigma^2_\infty(\phi)$.
Then, we use this expression to show how to define another Moran process approaching the same distribution $\mu_{\infty}$, with smaller or equal asymptotic square error.

In order to prove the asymptotic normality of the statistic $\sqrt{N} \big( m\big(\eta_T \big)(\phi) - \mu_T(\phi) \big)$, for every $T \ge 0$, we naturally need to ask, in addition to the law of large numbers established by Assumption \ref{assump:initial_condition}, for the existence of a central limit theorem on the initial empirical distribution, as stated in the following hypothesis.

\begin{customthm}{$(\mathrm{I'})$}[Asymptotic normality for initial empirical distribution]\label{assump:initial_condition_as_normality}
	For every $\phi \in \mathcal{B}_b(E)$, the empirical measure induced by the particle process at $t = 0$ satisfy the following condition:
	$\sqrt{N} \big( m(\eta_0)(\phi) - \mu_0(\phi) \big)$ converges in law towards a centred Gaussian distribution of variance $\mu_0(\phi^2)$, when $N \rightarrow \infty$.
\end{customthm}

Analogously to Lemma \ref{lemma:control_initial_condition}, we have that Assumption \ref{assump:initial_condition_as_normality} is verified when initially the $N$ particles are sampled independently according to $\mu_0 \in \mathcal{M}_1(E)$.
The proof of this result is a consequence of the classical functional central limit theorem for martingales.

\begin{theorem}[Asymptotic normality]\label{thm:quadbound}
	Suppose that Assumptions \ref{assump:initial_condition}, \ref{assump:initial_condition_as_normality}, \ref{assump:additive+symmetric} and \ref{assump:ergod_normalised} are verified. 
	Then, for every $\phi \in \mathcal{B}_b(E)$ and $T \ge 0$, we have that
	\(
	\sqrt{N} \big( m(\eta_T)(\phi) - \mu_T(\phi) \big)
	\)
	converges in law, when $N$ goes to infinity, towards a Gaussian centred random variable of variance
	\[
	\sigma^2_T(\phi) = \operatorname{Var}_{\mu_{T}}(\phi) + \int_0^T  S_{\mu_s} \big( W_{s,T}(\bar{\phi}_T) \big) \mathrm{d}s
	+ 2 \int_0^T  \mu_s \Big(  W_{s,T}(\bar{\phi}_T)^2 \Big( V^\mathrm{b}_{\mu_s} + \mu_{s}\big(V^\mathrm{d}_{\mu_s}\big) \Big) \Big) \mathrm{d}s,
	\]
	where $\operatorname{Var}_{\mu_{T}}$ stands for the variance with respect to $\mu_T$, 
	$\bar{\phi}_T := \phi - \mu_T(\phi)$
	and $S_{\mu}$ and $W_{t,T}$ are as defined in \eqref{eq:defSmu} and \eqref{eq:defWtT_intro}, respectively.
	Moreover,
	\begin{align*}
		\sigma^2_\infty(\phi) = & \operatorname{Var}_{\mu_{\infty}}(\phi)  
		+ 
		\int_0^\infty \mathrm{e}^{- 2 \lambda s} S_{\mu_\infty} \big( P_s^\Lambda(\bar{\phi}_\infty) \big) \mathrm{d}s + 2 \int_0^\infty \mathrm{e}^{-2 \lambda s} \mu_\infty\Big(  P_s^{\Lambda}(\bar{\phi}_\infty)^2 \Big( V^\mathrm{b}_{\mu_\infty} + \mu_{\infty}\big(V^\mathrm{d}_{\mu_\infty}\big) \Big) \Big) \mathrm{d}s,
	\end{align*}
	where $\operatorname{Var}_{\mu_{\infty}}$ stands for the variance with respect to $\mu_\infty$, also 
	$\bar{\phi}_\infty := \phi - \mu_{\infty}(\phi)$ and
	$\lambda$ is the eigenvalue in the statement of Lemma \ref{corollary:exp_ergodicity_nonnormalised}.
\end{theorem}
The proof of Theorem \ref{thm:quadbound} can be found in Section \ref{thm:proof_thm6}.
Note that the two integrals in the expression of $\sigma^2_\infty(\phi)$ in Theorem \ref{thm:quadbound} converge as a consequence of Lemma \ref{corollary:exp_ergodicity_nonnormalised}.

Let us mention the relation between Theorem \ref{thm:quadbound} and some existing results in the literature.
When $V^\mathrm{s}$ is null, and the selection rates do not depend on $\mu$, our result is related to Proposition 3.7 in \cite{MR1956078}.
Indeed, for the particular choice of the parameters of the model in \cite{MR1956078} as follows: $V = 2 V^\mathrm{b}$, $V' = 2 V^\mathrm{d}$ and $\rho = 1/2$, Theorem \ref{thm:quadbound} can be obtained from Proposition 3.7 in \cite{MR1956078}.
See also the multivariate fluctuation study in \cite[\S\ 3.3.2]{DelMoralMiclo2000}. 

Moreover, when $V_\mu^\mathrm{b}$ and $V_\mu^\mathrm{s}$ are null and thus $\Lambda = - V^\mathrm{d} \le 0$, we get
\[
\sigma^2_\infty(\phi) = \operatorname{Var}_{\mu_{\infty}}(\phi)  - 2 \lambda \int_0^\infty \mathrm{e}^{- 2  \lambda s}  \operatorname{Var}_{\mu_\infty} \left( P^\Lambda_s(\phi) \right) \mathrm{d}s.
\]
When the process $(\eta_t)_{t \ge 0}$ is ergodic and converges in law to some random variable $\eta_\infty$, when $t \to \infty$, Theorem \ref{thm:quadbound} states that $\sqrt{N} \big(m(\eta_\infty)(\phi) - \mu_\infty(\phi)\big)$ converges to a centred Gaussian law of variance $\sigma^2_\infty(\phi)$, when $N\to \infty$. 
Indeed, recall that a Gaussian sequence converges in law if their first two moments converge. 
In particular, we recover (and extend) the result of Lellièvre et al.\ \cite[Thm.\ 2.4]{Lelievre2018} for finite state spaces.
Note that the negative constant $\lambda$ in the previous expression for $\sigma^2_\infty(\phi)$ is the opposite of that in \cite[Thm.\ 2.4]{Lelievre2018}.

The expression for $\sigma^2_\infty(\phi)$, when $V_\mu^\mathrm{s} = 0$, is also similar to the expression for the asymptotic square error in Theorem 4.4 in \cite{MR2262944}.
However, the results in \cite{MR2262944} do not include the asymptotic normality we prove in Theorem \ref{thm:quadbound}.
See also Corollary 2.7 and Remark 2.8 \cite{Cerou2020} for a central limit theorem for the empirical measure induced by Fleming\,--\,Viot particle systems.

Note that the three summands in the expression of $\sigma^2_T(\phi)$ in Theorem \ref{thm:quadbound} are positive, for every $T \ge 0$.
Moreover, the limit $(\mu_t)_{t \ge 0}$ is invariant by the choice of the symmetric component $V_\mu^\mathrm{s}$ in Assumption \ref{assump:additive+symmetric}.
As a consequence, for a given selection rate $V_\mu$ we can obtain another Moran process approaching the same limit distribution taking the selection rate $V_\mu - \Sigma_\mu \ge 0$, where $\Sigma_\mu$ is a symmetric function in $\mathcal{B}_b(E \times E)$.
We thus get the following result.

\begin{corollary}[Moran process with smaller asymptotic square error]\label{cor:small_asymp_error}
	Suppose that Assumptions \ref{assump:initial_condition}, \ref{assump:initial_condition_as_normality}, \ref{assump:additive+symmetric} and \ref{assump:ergod_normalised} are verified. 
	Let $(\eta_t)_{t \ge 0}$ and $(\eta_t^\star)_{t \ge 0}$ be the Moran processes with the same mutation rates and selection rates given by $V_\mu$ and $V_\mu - \Sigma_\mu$, respectively, where
	\[
	\Sigma_\mu(x,y) :=  \min\Big\{ V_\mu^{\mathrm{d}}(x), V_\mu^{\mathrm{b}}(x) \Big\} \pmb{1}_{ \{x\} } + \min\Big\{ V_\mu^{\mathrm{d}}(y), V_\mu^{\mathrm{b}}(y) \Big\} \pmb{1}_{ \{y\} } + V_\mu^{\mathrm{s}}(x,y),
	\]
	where $\pmb{1}_A$ stands for the indicator function on $A \subset E$.
	Then,
	\[
	\lim\limits_{N \rightarrow \infty}  N \mathbb{E} \left[ \big( m(\eta_T^\star)(\phi) - \mu_{T}(\phi) \big)^2 \right]
	\le
	\lim\limits_{N \rightarrow \infty}  N \mathbb{E} \left[ \big( m(\eta_T)(\phi) - \mu_{T}(\phi) \big)^2 \right],
	\]
	for all $T \ge 0$.
\end{corollary}

Note that the selection rate $V_\mu-\Sigma_\mu$ in the statement of Corollary \ref{cor:small_asymp_error} satisfies Assumption \ref{assump:additive+symmetric}.
The proof of the previous result is thus a simple consequence of Theorem \ref{thm:quadbound}.
In Example \ref{example:twi_sites}, we discuss the application of this result to the simple case of the bi-allelic Moran model, that is, when the cardinality of $E$ is $2$.

\subsection{Examples}\label{sec:examples}

In this section we consider several examples where Assumption \ref{assump:ergod_normalised} holds, for the process with additive selection satisfying \ref{assump:additive+symmetric} and \ref{assump:gral_slection rate}.

The first example we consider is precisely when $E = \{1,2\}$.
This example offers us the opportunity to compare our result with the existing results on bi-allelic Moran models and the Fleming\,--\,Viot particle process approximating the QSD of an absorbing Markov chain with two transient states.  

\begin{example}[Two-allelic Moran model]\label{example:twi_sites}
	Consider the two-allelic Moran model on $E = \{1,2\}$ with mutation rate matrix
	\[
	Q = \left(
	\begin{array}{rr}
		-a &  a\\
		b & -b
	\end{array}
	\right)	
	\]
	and selection rates $V_{1,2} = p$ and $V_{2,1} = q$, with $a,b > 0$ and $p,q \ge 0$.
	Let us assume, without loss of generality, that $p \le q$. 
	
	The empirical probability measure induced by this Moran process approaches the QSD of the absorbing Markov chain on $E \cup \{ \partial \}$, where $\partial$ is an absorbing state, with infinitesimal generator
	\[
	\left(
	\begin{array}{ccc}
		-(a + p) &  a & p \\
		b & -(b + q) & q \\
		0 & 0 & 0
	\end{array}
	\right).
	\]
	See \cite{cordero_deterministic_2017} and \cite[\S\ 3]{cloez_fleming-viot_2016} for a deeper treatment of this model and the limit behaviour of the interacting particle process approaching its QSD.
	
	Theorem \ref{thm:propagation_chaos} applied in this case improves Proposition 3.1 in \cite{cordero_deterministic_2017} and Theorem 3.1 (see also Remark 3.2) in \cite{2015Benaim&Cloez}.
	Furthermore, Theorem \ref{thm:uniformLp}, and also \eqref{eq:conseq_of_main_thm_st_distrib}, improve the control of the speed of convergence to stationarity of the bounds obtained in \cite[Cor.\ 1.5]{cloez_quantitative_2016} and \cite[Thm.\ 2.4]{zbMATH07298000}.

	Likewise, as a consequence of Corollary \ref{cor:small_asymp_error}, we have that the Moran model with the same mutation rate matrix $Q$ and with selection rates $V_{1,2} = 0$ and $V_{2,1} = q-p$, approaches the same QSD but with smaller asymptotic square error.
	
	Consider now the case when $p = q$.
	Then, the empirical distribution induced by the particle system approaches the stationary distribution of the process generated by $Q$.
	When $E$ is finite, the results about the spectrum of the generator $\mathcal{Q}$ in \cite{2020arXiv201008809C} imply that the asymptotic ergodicity is independent of the value of $p$.
	Besides, Corollary \ref{cor:small_asymp_error} implies that a minimal asymptotic variance is obtained when there is no selection, that is, when the particle system is simply given by $N$ independent particles, where each of them is driven by $Q$.
\end{example}

We now focus on the classical birth and death Markov chain.
The existence and uniqueness of QSD for these models have been well understood.
We rely on existing results to find explicit conditions on the parameters of the birth and death chain that are equivalent to the existence of a unique QSD and the uniform exponential convergence.

\begin{example}[Birth and death chain]\label{example:B-D_chain}
	
	Consider two positive sequences $(b_x)_{x \ge 1}$ and $(d_x)_{x \ge 1}$ and the Markov chain on $\mathbb{N}$ with rate matrix
	\[
	Q_{x,y} := \left\{
	\begin{array}{cl}
		b_x & \text{ if } x \ge 1 \text{ and } y = x + 1 \\
		d_x & \text{ if } x \ge 2 \text{ and } y = x - 1 \\
		0 & \text{ otherwise,}
	\end{array}
	\right.	
	\]
	and $\Lambda := d_1 \mathbf{1}_{\{x = 1\}}$.
	Note that $0$ is an absorbing state and  $\mathbb{N}$ is a transient class. 
	Van Doorn \cite[Thm.\ 3.2]{vanDorn1991} has found an explicit condition characterising the three possible cases: there is no QSD, there exists a unique QSD or there exists an infinite continuum of QSDs.
	See also \cite[\S\ 4]{meleard_quasi-stationary_2012}.
	Furthermore, Mart\'inez et al.\ \cite[Thm.\ 2]{martinez2014} have proved that the existence of a unique QSD is in fact equivalent to the uniform exponential convergence of the law of the conditioned process to its QSD, i.e.\ Assumption \ref{assump:ergod_normalised}.
	In addition, this occurs if and only if
	\begin{equation}\label{condition:uniqueQSD_BDchains}
		\sum\limits_{k \ge 2} \frac{1}{d_k \alpha_k} \sum\limits_{r \ge k} \alpha_r < \infty,
	\end{equation}
	where $\alpha_r := \prod\limits_{i = 1}^{r-1} b_i \bigg/ \prod\limits_{i = 2}^{r} d_i$.
	We refer also to \cite[\S\ 4.1]{zbMATH06540706}, where the uniform exponential convergence is ensured for some generalisations of the classical birth and death chain.
\end{example}

We end this section presenting two quantitative criteria on the transition rates and on the spectral elements, respectively, ensuring the uniform exponential convergence in \ref{assump:ergod_normalised}.

\begin{example}[A criterion on the mutation and selection rates]
	
	We next describe a criterion on the transition rates, which is Theorem 3 in \cite{martinez2014}.
	Assume that \ref{assump:additive+symmetric} is verified, and the following condition holds: there exists a finite subset $K \subset E$ such that
	\[
	\inf\limits_{y \in E \setminus K} \left( \Lambda(y) + \sum_{x \in K} Q_{y,x} \right) > \sup_{y \in E} \Lambda(y).
	\]
	Then, \ref{assump:ergod_normalised} holds.
	This provides an easy condition on the mutation rates and $\Lambda$ to verify Assumption \ref{assump:ergod_normalised}, which is applicable to a wide range of Moran processes with discrete countable state space.
	See also \cite[Thm.\ 1.1]{cloez_quantitative_2016}, where a stronger condition is asked in order to provide, via a coupling technique, explicit constants for the upper bound in \ref{assump:ergod_normalised}.
	
\end{example}

\begin{example}[A spectral criterion]
	Assume there exists a triplet $(\mu_{\infty}, h, \lambda) \in \mathcal{M}_1(E) \times \mathcal{B}_b(E) \times \mathbb{R}$, of eigenelements of $Q+ \Lambda$ such that $\lambda$ is an eigenvalue of $Q + \Lambda$, $h$ is strictly positive, $\mu_{\infty}(h) = 1$, and
	\[
	\mu_{\infty} P_t^\Lambda = \mathrm{e}^{\lambda t} \mu_{\infty} \text{ and } P_t^\Lambda (h) = \mathrm{e}^{\lambda t} h.
	\] 
	Note that these are the eigenelements in the statement of Lemma \ref{corollary:exp_ergodicity_nonnormalised}.
	Let us also assume
	\(
	\|h^{-1}\| \le \infty,
	\)
	which is always true if $E$ is finite, and furthermore, there exists $\epsilon > 0$ such that the set
	\[
	K_\epsilon := \{ x \in E: \Lambda(x) \ge \lambda - \epsilon \}	
	\]
	is finite. 
	Then, \ref{assump:ergod_normalised} is verified.
	
	The proof is based on the methods in \cite{downetal1996}, and is very similar to the proofs of Proposition 3.2 in \cite{MR2262944} and Proposition A.5 in \cite{Angeli2021}.
	Consider the Doob's $h$-transform 
	\[
	P_ t^{\Lambda, h} := \frac{1}{h} \mathrm{e}^{- \lambda t} P_t^{\Lambda} (h \cdot),	
	\]
	which is the semigroup associated to an irreducible continuous-time Markov chain on $E$ with generator $Q^h$ acting on every $\phi \in \mathcal{M}_1(E)$ as follows
	\[
	Q^h (\phi) = \frac{1}{h} \big( Q + \Lambda - \lambda \big)(h \phi).	
	\]
	Furthermore, the process driven by $Q^h$ has stationary distribution $\mu_\infty^h \in \mathcal{M}_1(E)$, satisfying $\mu_\infty^h(\phi) = \mu_{\infty}(h \phi)$, for every $\phi \in \mathcal{B}_b(E)$.
	Now, note that $h^{-1}$ is bounded on $K_{\epsilon}$, and consequently there exists $\beta > 0$ such that
	\[
	Q^h (h^{-1}) = \frac{\Lambda - \lambda}{h} \le - \epsilon h^{-1} + \beta \mathbb{1}_{K_\epsilon}.
	\]
	Thus, condition $(\tilde{D})$ in \cite{downetal1996} is verified and using their Theorem 5.2-(c) we get the $h^{-1}$-uniform exponential ergodicity of $P_t^{\Lambda, h}$ as follows
	\[
	\sup_{|g| \le h^{-1}} \left| \frac{1}{h(x)} \mathrm{e}^{- \lambda t} \delta_x P_t^{\Lambda}(h g) - \mu_\infty(h g) \right| \le \frac{C}{h(x)} \rho^{t}.
	\]
	Multiplying by $h(x)$ the previous inequality and taking $\phi = h g \in \mathcal{B}_1(E)$, we get the uniform exponential ergodicity \eqref{eq:assump_ergo_FeynmanKac}.
	Finally, it is not difficult to verify that Assumption \ref{assump:ergod_normalised} also holds, using the exponential ergodicity \eqref{eq:assump_ergo_FeynmanKac} and the fact that $\|h^{-1}\| < \infty$.
\end{example}

\begin{remark}
	In \cite[Appendix]{Angeli2021}, the authors state a similar result, but they do not include the fact that $\|h^{-1}\|$ is bounded in their hypothesis.
	We have not found or understood the  arguments making the authors claim that $\|h^{-1}\|$ is bounded when the state space is locally compact \cite[p.\ 150]{Angeli2021}.
	We next provide an example of a birth and death chain whose generator allows an unbounded eigenfunction associated to its greatest eigenvalue.
	Indeed, let us consider the following parameters for the birth and death chain in Example \ref{example:B-D_chain}: 
	$b_i = b$, and $d_i = d$, for all $i \ge 2$, with $b< d$. 
	Moreover, take $b_1 > b$ and $d_1 = d(\mathrm{e} - 1)$.
	Hence, taking $h : n \in \mathbb{N} \mapsto \mathrm{e}^{-n}$ we get $\big( Q + \Lambda \big) (h) = \lambda h$, for $\lambda = b(\mathrm{e}^{-1} - 1 ) + d (\mathrm{e} - 1) > 0$.
	Moreover, $K_\epsilon = \{1\}$ is finite (compact), but $\|r^{-1}\| = \infty$.
	In fact, the infinite sum \eqref{condition:uniqueQSD_BDchains} diverges, thus this birth and death chain allows an infinite number of QSDs (cf.\ \cite[Thm.\ 3.2]{vanDorn1991}).
\end{remark}

\subsection{Final remarks}

The fact that the state space $E$ is countable and discrete is not necessary for our proofs. 
Therefore, we expect to be able to extend all our results to more general Markov processes following the same methods.

There are lots of possible directions to continue this research.
Maybe, the more natural is to weaken the condition  \ref{assump:ergod_normalised} and consider the case where there exists a minimal QSD but the exponential convergence is not uniform on $\mathcal{M}_1(E)$.
Lots of research have been done for controlling the domains of attraction of the minimal QSD.
See for example the works of Champagnat and Villemonais \cite{Villemonais2015B&D,ChampagnatVillemonaisP2017a,Villemonais2020Rpositive,Villemonais2020Rpositive:Erratum,zbMATH07298000}
and also the related works of Bansaye et al.\ \cite{Bansaye2019,zbMATH07181472}, and the references therein.
Another interesting research direction is to improve the upper bound constants in Theorem \ref{thm:propagation_chaos}.
In this sense, the results of Arnaudon and Del Moral \cite[Thm.\ 5.10 and Cor.\ 5.12]{MR4193898} suggest that a bound of type $C_p T$ could hold.
Hence, a future research direction would be to combine the approach of \cite{MR4193898} and this paper to improve the upper bound in Theorem \ref{thm:propagation_chaos}.
Moreover, the results in \cite[\S\ 5]{MR4193898} could also be useful to obtain exponential concentration inequalities, which is a natural continuation of the research on the long time behaviour of the empirical measure induced by Moran type particle processes.

\section{Proof of the main results}\label{sec:proofs}

This section contains the proofs of Theorems \ref{thm:propagation_chaos} and \ref{thm:quadbound}, which are, in our opinion, the main results in the present article with original proof methods.
In Section \ref{sec:martingale} we briefly study the martingale problem associated to the generator $\mathcal{Q}_N$.
Then, in Sections \ref{sec:proofThm1} and \ref{thm:proof_thm6} we prove Theorems \ref{thm:propagation_chaos} and \ref{thm:quadbound}, respectively.
Other comments, results and proofs are deferred to the appendices.

\subsection{The associated martingale problem}\label{sec:martingale}

For a Markovian generator $L$, its associated ``carr\'e-du-champ'' operator, denoted $\Gamma_L$, is defined by 
\[
\Gamma_L : \phi \mapsto L (\phi^2) - 2 \phi L \phi.
\]
See, for example, \cite[Def.\ 2.5.1]{zbMATH01633816} for more details on the theory related to this operator.

It is not difficult to prove that $\Gamma_{\mathcal{Q}}$ satisfies
\[
\Gamma_{\mathcal{Q}} (\psi)(\eta) = \sum_{x \in E} \eta(x) \sum_{y \in E} \left( Q_{x,y} + V_{m(\eta)}(x,y) \frac{\eta(y)}{N} \right) [\psi(\eta - \mathbf{e}_x + \mathbf{e}_y) - \psi(\eta)]^2,
\]
where for every $\eta \in \mathcal{E}_N$.
We recall that $m(\eta)$ denotes the empirical distribution induced by $\eta \in \mathcal{E}_N$.
Moreover, $m(\eta)(\phi)$ stands for the mean of $\phi$ with respect to $m(\eta)$, for every $\phi \in \mathcal{B}_b(E)$.
Suppose that one of Assumptions \ref{assump:gral_slection rate} or \ref{assump:additive+symmetric} is verified.
In either case, let us denote $\widetilde{V}_\mu := V_\mu - V^\mathrm{s}_\mu$.

\begin{lemma} \label{lemma:1}
	Suppose that one of Assumptions \ref{assump:gral_slection rate} or \ref{assump:additive+symmetric} is verified.
	We have
	\begin{align*}
		\mathcal{Q}(m(\cdot) (\phi)) &= m(\cdot) \left( \widetilde{Q}_{m(\cdot)} (\phi) \right),\\
		\Gamma_{\mathcal{Q}} \big(m(\cdot)(\phi)\big) &= \frac{1}{N} m(\cdot) \left( \Gamma_{Q_{m(\cdot)}} (\phi) \right),
	\end{align*}
	where
	\begin{align*}
		\widetilde{Q}_{\mu} \phi: x &\mapsto (Q \phi)(x) + \sum_{y \in E} \mu(y) \widetilde{V}_\mu(x,y) [\phi(y)-\phi(x)],\\
		Q_{\mu} \phi: x &\mapsto (Q \phi)(x) + \sum_{y \in E} \mu(y) V_{\mu}(x,y) [\phi(y)-\phi(x)],
	\end{align*}
	for every $\phi \in \mathcal{B}_b(E)$ and all $x \in E$.
\end{lemma}

The proof of Lemma \ref{lemma:1} is deferred to Appendix \ref{app:proof_lemma:1}.

Using this result, we can study the martingale problem associated to the process $\left( m(\eta_t) (\psi_t) \right)_{t \ge 0}$. 

\begin{proposition}[Martingale decomposition] \label{prop:martingale}
	Let $\psi$ be a function on $E \times \mathbb{R}_+$ such that $\psi_\cdot(x)$ is continuously differentiable in $\mathbb{R}_+$, for every $x \in E$ and $\psi_t(\cdot) \in \mathcal{B}_b(E)$, for every $t \in \mathbb{R}_+$.
	Then, the process $\big( \mathcal{M}_t(\psi_\cdot) \big)_{t \ge 0}$ such that
	\[
	\mathcal{M}_t(\psi_\cdot)  := m(\eta_{t})(\psi_{t}) - m(\eta_0)(\psi_0) - \int_0^t m(\eta_{s}) \left(\partial_s \psi_s + \widetilde{Q}_{m(\eta_{s})} (\psi_{s})  \right) \mathrm{d}s,
	\]
	where $\widetilde{Q}_{\mu}$ is defined as in Lemma \ref{lemma:1},
	is a local martingale, with \emph{predictable quadratic variation} given by
	\begin{align*}
		\langle \mathcal{M}(\psi_\cdot) \rangle_t &= \frac{1}{N} \int_0^t m(\eta_s) \Big( \Gamma_{Q_{m(\eta_s)}}(\psi_s) \Big) \mathrm{d}s. 
	\end{align*}
	Moreover,
	\[
	|\Delta \mathcal{M}_t(\psi_t)| \le \frac{2 \|\psi_t\|}{N}.
	\]
\end{proposition}

The proof of Proposition \ref{prop:martingale} is deferred to Appendix \ref{app:proof_prop:martingale}.

Now, for a function $\psi$ on $E \times \mathbb{R}_+$, continuously differentiable in $\mathbb{R}_+,$ we get
\begin{equation*}
	\mathrm{d} m(\eta_t) (\psi_t) = \mathrm{d} \mathcal{M}_t(\psi_\cdot) + m(\eta_t) \Big( \partial_t \psi_t + \widetilde{Q}_{m(\eta_t)} (\psi_t) \Big) \mathrm{d} t.
\end{equation*}
Thus, the empirical measure induced by the particle process is a perturbation of the dynamic given by \eqref{eq:preODE}, by a martingale whose jumps and predictable quadratic variation are of order ${1}/{N}$.

\subsection{Proof of Theorem \ref{thm:propagation_chaos}} \label{sec:proofThm1}

Throughout this section, we will suppose that the expression for the selection rates in Assumption \ref{assump:gral_slection rate} is verified.
We will denote $\widetilde{Q}_\mu = Q + \widetilde{\Pi}_\mu$ as in Lemma \ref{lemma:1}, namely
\[
\widetilde{\Pi}_{\mu} \phi: x \mapsto \sum_{y \in E} \mu(y) V(x,y) [\phi(y)-\phi(x)],
\]
where 
\(
V = V_\mu - V_\mu^\mathrm{s},
\)
which is independent of $\mu \in \mathcal{M}_1(E)$.

The family of generators $\big(\widetilde{Q}_{\mu_t}\big)_{t\ge 0}$ defines an inhomogeneous-time Markov chain, which is associated to a map $(s,t) \mapsto P(s,t)$, for all $s \le t$ such that $P(s,s) = I$, for all $s \ge 0$ and satisfies the forward and backward Kolmogorov equations:
\begin{align}
	\partial_t P(s,t) &= P(s,t) \widetilde{Q}_{\mu_t}, \text{ for } t \ge s, \label{eq:forward_Kolmogorov}\\
	\partial_s P(s,t) &= - \widetilde{Q}_{\mu_s} P(s,t), \text{ for } s \le t. \nonumber 
\end{align}
See \cite{2014NonHomogeneous} and the references therein.
Besides, using the forward Kolmogorov equation \eqref{eq:forward_Kolmogorov}, we get that $(\mu_t)_{t \ge 0}$ as in Lemma \ref{assump:gral_slection rate_existence}, satisfies the propagation equation 
\( 
\mu_T = \mu_t P(t,T). 
\)
Note that since $P(t,T)$ is the propagator of an inhomogeneous Markov chain, we get $\|P(t,T)\| \le 1$ for all $t \in [0,T]$, which implies
\begin{equation}\label{eq:bound_integral_WtT}
	\int_0^T \|P(s,T)(\phi)\|^{p} \mathrm{d}s \le  T.  
\end{equation}

Let us now study the control in $\mathbb{L}^p$ norm for the martingales that are obtained taking the functions $t \in [0,T] \mapsto P(\cdot, T)\big(\phi\big)$ and $t \in [0,T] \mapsto P(\cdot, T)\big(\phi\big)^2$ in Proposition \ref{prop:martingale}.	
Note that,
\[
\partial_t \Big( P(t,T)\big(\phi\big)\Big) = -  \widetilde{Q}_{\mu_t} \Big( P(t,T)\big(\phi\big) \Big).
\]
From Proposition \ref{prop:martingale}, we get the following local martingale for $t \in [0,T]$:
\begin{align*}
	\mathcal{M}_t \Big( P(\cdot, T) \big(\phi\big) \Big) :=& 
	\; m(\eta_t)\Big( P(t,T)\big(\phi\big) \Big) - m(\eta_0)\Big( P(0,T)\big(\phi\big) \Big) \\
	&- \int_0^t m(\eta_s) \Big(  \widetilde{Q}_{m(\eta_s)} \Big(P(s,T)\big(\phi\big) \Big) - \widetilde{Q}_{\mu_s} \Big(P(s,T)\big(\phi\big) \Big) \Big) \mathrm{d}s.
\end{align*}
Similarly, we get
\begin{align}
	\mathcal{M}_t \left( P(\cdot, T)\big(\phi\big)^2 \right) :=& \, 
	m(\eta_t)\left( P(t,T)\big(\phi\big)^2 \right) - m(\eta_0)\left( P(0,T)\big(\phi\big) \right) \nonumber \\
	&- \int_0^t m(\eta_s)  \left(  \widetilde{Q}_{m(\eta_s)} \Big(P(s,T)\big(\phi\big)^2 \Big) - 2 P(s,T) (\phi) \; \widetilde{Q}_{\mu_s} \Big( P(s,T)\big(\phi\big) \Big) \right)  \mathrm{d}s. \label{eq:martingalePtT2}
\end{align}
Moreover, by definition, $P(T,T)\big(\phi\big) = \phi$.

\begin{lemma}[Control of the predictable quadratic variation]\label{lemma:bound_quadratic_variation:additive}
	Assume that Assumption \ref{assump:gral_slection rate} is verified.
	For every test function $\phi \in \mathcal{B}_1(E)$, we have
	\begin{align*}
		N \Big\langle \mathcal{M}\left(P(\cdot, T)\big( \phi \big) \right) \Big\rangle_t &\le C ( t + 1) - \mathcal{M}_t \Big(P(\cdot, T)\big( \phi \big)^{2}\Big), \;\; \text{ for all } \;\;  t \in [0,T].
	\end{align*}
\end{lemma}

The proof of Lemma \ref{lemma:bound_quadratic_variation:additive} is postponed to Appendix \ref{app:proof_lemma:bound_quadratic_variation:additive}.

The following lemma is a generalisation of the classical Burkholder\,--\,Davis\,--\,Gundy (BDG) inequality \cite[Thm.\ 20.12]{Kallenberg2021}.
The lower bound is obtained from the classical BDG inequality.
The proof of the upper bound can be found in \cite[Lemma 6.2]{MR2262944}.

\begin{lemma}[BDG inequalities]\label{lemma:BDGineq} 
	Let $\left(\mathcal{M}_t\right)_{t \ge 0}$ be a quasi-left-continuous (i.e.\ with continuous predictable increasing process) locally square-integrable martingale with $M_0 = 0$ and bounded jumps $$\sup\limits_{0 \le t \le T} |\Delta \mathcal{M}_t| \le a < + \infty.$$
	Then, there exists a constant $C$, possibly depending on $q$, such that
	\[
	\mathbb{E}\left[ \left(\sup_{t \in [0,T]} \mathcal{M}_t \right)^{2^{q+1}} \right] \le C \mathbb{E} \left[ \Big( [\mathcal{M}]_T \Big)^{2^q} \right] \le C \sum_{k = 0}^q a^{2^{q+1} - 2^{k+1}} \mathbb{E} \left[ \Big( \langle \mathcal{M} \rangle_T \Big)^{2^k} \right].
	\]
\end{lemma}

We are now in position to establish a control of the quadratic variation of the martingale $\left(\mathcal{M}_t\Big( P(\cdot, T)\big( \phi \big) \Big) \right)_{t \in [0,T]}$.
\begin{lemma}[Control of the quadratic variation]\label{thm:control_quadratic_variation:additive}
	Assume that Assumption \ref{assump:gral_slection rate} is verified.
	For all $p > 0$ and all test function $\phi \in \mathcal{B}_1(E)$, there exists a positive $C_p$ (possibly depending on $p$) such that
	\[
	\mathbb{E} \left[ \left( \left[ \mathcal{M}\Big( P(\cdot, T)\big( \phi \big) \Big) \right]_t \right)^p \right] \le \frac{C_p (t + 1)^p}{N^p}, \;\; \text{ for all } \;\;  t \in [0,T]
	\]
\end{lemma}
The proof of this result is inspired by the proof of Theorem 5.4 in \cite{MR2262944}, and it is deferred to Appendix \ref{app:proof_thm:control_quadratic_variation:additive}.

\begin{proof}[Proof of Theorem \ref{thm:propagation_chaos}]
	
	Let us denote $\psi_{s,T} := P(s,T)(\bar{\phi}_T)$, which satisfies the backward Kolmogorov equation \eqref{eq:forward_Kolmogorov}.
	We have that $\big(\mathcal{M}_t\big(\psi_{\cdot,T}\big)\big)_{t \in [0,T]}$, defined as in Proposition \ref{prop:martingale}, is a local martingale. 
	Moreover,
	\begin{align}
		\mathcal{M}_{T}(\psi_{\cdot,T}) &= m(\eta_{T})(\psi_{T,T}) - m(\eta_0)(\psi_{0,T}) - \int_0^T m(\eta_{s}) \left(- \widetilde{Q}_{\mu_s} (\psi_{s,T}) + \widetilde{Q}_{m(\eta_{s})} (\psi_{s,T})  \right) \mathrm{d}s \nonumber \\
		&= m(\eta_{T})(\phi) - \mu_T(\phi) - m(\eta_0)(\psi_{0,T})  - \int_0^T m(\eta_{s}) \left(\widetilde{\Pi}_{m(\eta_{s})} (\psi_{s,T}) - \widetilde{\Pi}_{\mu_s} (\psi_{s,T})  \right) \mathrm{d}s. \label{eq:martingale_decomposition_proof}
	\end{align}
	
	Note that for any two probability measures $\lambda$ and $\mu$ on $E$ and for every function $\psi$ on $E$ we have
	\begin{equation}\label{eq:transposingPi}
		\lambda \big( \widetilde{\Pi}_{\mu} (\psi) \big) = -\mu \big( \widehat{\Pi}_{\lambda} (\psi) \big),
	\end{equation}
	where $\widehat{\Pi}_{\lambda}$ acts on a test function $\psi$ as follows:
	\[
	\widehat{\Pi}_{\lambda} (\psi): x \mapsto \sum_{y \in E} \lambda(y) V(y,x) [\psi(y) - \psi(x)].
	\]
	Now, using \eqref{eq:martingale_decomposition_proof} and \eqref{eq:transposingPi} we get
	\begin{equation} \label{eq:martingale_general_proof}
		\mathcal{M}_{T}(\psi_{\cdot,T}) = m(\eta_{T})(\phi) - \mu_T(\phi) - m(\eta_0)(\psi_{0,T})  + \int_0^T \big(m(\eta_{s}) - \mu_s\big) \left(\widehat{\Pi}_{m(\eta_{s})} (\psi_{s,T})  \right) \mathrm{d}s,
	\end{equation}
	where $\psi_{s,T} = P(s,T)(\bar{\phi}_T)$.
	
	Hence, using \eqref{eq:martingale_general_proof}, we can ensure the existence of a positive constant $C_p > 0$, depending on $p$, such that
	\begin{align*}
		\sup_{t \le T} \left| m(\eta_t)(\phi) - \mu_t(\phi) \right|^p &\le C_p \Big( | m(\eta_0)( \psi_{0,T}) - \mu_0(\psi_{0,T}) |^p  +  \sup_{t \le T} |\mathcal{M}_t( \psi_{\cdot,T} )|^p  + R_{p}(T)  \Big),
	\end{align*}
	where
	\[
	R_p(T) =  \int_0^T \left| \big(m(\eta_{s}) - \mu_s \big) \left(\widehat{\Pi}_{m(\eta_{s})} (\psi_{s,T})  \right) \right|^p \mathrm{d}s.
	\]
	
	The initial error can be controlled using Assumption \ref{assump:initial_condition}.
	Indeed, there exists $C_p^{(1)} > 0$ such that
	\[
	\mathbb{E}\big[ | m(\eta_0)(\psi_{0,T})  - \mu_0 (\psi_{0,T})|^p \big] \le \frac{C_p^{(1)}}{N^{p/2}}.
	\]
	Furthermore, using Lemma \ref{thm:control_quadratic_variation:additive} and BDG inequality, we can ensure the existence of a positive constant $C_p^{(2)}$ such that
	\[
	\mathbb{E}\left[ \sup_{t \le T} \left|\mathcal{M}_t \Big(P(\cdot, t)\big(\bar{\phi}_T \big) \Big) \right|^p \right] \le \frac{C_p^{(2)} (T+1)^{p/2}}{N^{p/2}},
	\]
	for all $p \ge 1$.
	Let us denote by $\lambda_s$, the (random) signed measure
	\[
	\lambda_s := m(\eta_{s}) - \mu_s.
	\]
	We have
	\[
	R_p(T) \le C_p \left(   \int_0^T \left| \lambda_s \left(\widehat{\Pi}_{\mu_s} (\psi_{s,T})  \right) \right|^p \mathrm{d}s  +  \int_0^T \left| \lambda_s \left(\widehat{\Pi}_{ \lambda_s } (\psi_{s,T})  \right) \right|^p \mathrm{d}s  \right).
	\]
	The first term in the last expression can be controlled, since $\widehat{\Pi}_{\mu_s} (\psi_{s,T})$ is not random.
	Indeed,
	\begin{align*}
		I_1(T) &:=  \int_0^T \left| \big(m(\eta_{s}) - \mu_s\big) \left(\widehat{\Pi}_{\mu_s} (\psi_{s,T})  \right) \right|^p  \mathrm{d}s  \le 2^p \|V\|^p \int_0^T \left| \big(m(\eta_{s}) - \mu_s\big) \left( \frac{\widehat{\Pi}_{\mu_s} (\psi_{s,T})}{2\|V\|}  \right) \right|^p \mathrm{d}s.
	\end{align*}
		For the second term, note that
	\begin{align*}
		I_2(T) &:=  \int_0^T \left| \sum_{i \ge 1} \sum_{x,y \in E} \lambda_s (x) \lambda_s (y)  V^\mathrm{d}_i(x)V^\mathrm{b}_i(y) \Big[\psi_{s,T}(y) - \psi_{s,T}(x) \Big] \right|^p \mathrm{d}s  \\
		&= \int_0^T \left| \sum_{i \ge 1}
		\lambda_s(V_i^{\mathrm{d}} - V_i^{\mathrm{b}}) \lambda_s(V_i^{\mathrm{b}} \psi_{s,T}) + \lambda_s(V_i^{\mathrm{b}}) \lambda_s\Big((V_i^{\mathrm{b}} - V_i^{\mathrm{d}}) \psi_{s,T}\Big)
		 \right|^p \mathrm{d}s.
	\end{align*}
	Hence, using the Assumption \ref{assump:gral_slection rate}, we conclude that there exists a positive constant $C_p^{(3)}$, depending on $p$, such that
	\[
		I_2(T) \le C_p^{(3)}  \int_0^T  \left| (m(\eta_s) - \mu_s) \left( \frac{1}{\kappa} \sum_{i \ge 1} |V^\mathrm{d}_i - V^\mathrm{b}_i| \right) \right|^p + \left| (m(\eta_s) - \mu_s) \left( \frac{|\psi_{s,T}|}{\kappa} \sum_{i \ge 1} |V^\mathrm{d}_i - V^\mathrm{b}_i| \right) \right|^p \mathrm{d}s,
	\]
	where $\kappa = \left\| \sum\limits_{i \ge 1} |V^\mathrm{d}_i - V^\mathrm{b}_i| \right\| < \infty$.
	
	Let us define
	\[
	\Phi_p(t) := \sup\limits_{ \phi \in \mathcal{B}_1(E) }  \mathbb{E}\left[ \sup_{s \le t} \left| m(\eta_s)(\phi) - \mu_s(\phi) \right|^p \right].
	\]
	Thus, taking the expectations of $I_1(T)$ and $I_2(T)$, we can ensure the existence of a positive constant $C_p^{(4)}$ such that
	\begin{align*}
		\mathbb{E}[R_p(T)] \le C_p^{(4)} \int_0^T \Phi_p(s) \mathrm{d}s.
	\end{align*}
	
	Hence, there exists $K_{p} > 0$ such that
	\begin{align*}
		\Phi_p(T) &\le \frac{K_{p} \left(1 + (1 + T)^{p/2}\right)}{N^{p/2}} + C_p^{(4)} \int_0^T \Phi_p(s) \mathrm{d}s,
	\end{align*}
	which, using Gr\"{o}nwall inequality, ensures the existence of two constants $\alpha_p$ and $\beta_p$ (possibly depending on $p$) such that
	\[
	\Phi_p(T)^{1/p} \le \alpha_p \frac{\sqrt{1 + T}}{\sqrt{N}} \mathrm{e}^{\beta_p T}.
	\]
\end{proof}

\begin{proof}[Proof of Corollary \ref{cor:convergence_in_normL}]
	Let $(x_n)_{n \ge 1}$ be the enumeration of the elements in $E$, in the definition of the distance $\| \cdot \|_\mathrm{w}$.
	Note that
	\begin{align*}
		\mathbb{E}\left[ \sup_{t \in [0,T]} \big( \|m(\eta_t) - \mu_t \|_\mathrm{w} \big)^p \right]^{1/p} &= \mathbb{E} \left[ \sup_{t \in [0,T]} \left( \sum_{k \ge 1} 2^{-k} |m(\eta_t)(x_k) - \mu_t(x_k) | \right)^p \right]^{1/p} \\
		&\le \sum_{k \ge 1} 2^{-k} \mathbb{E} \left[ \sup_{t \in [0,T]} |m(\eta_t)(x_k) - \mu_t(x_k) |^p \right]^{1/p}\\
		&\le \alpha_p \frac{\sqrt{1 + T}}{\sqrt{N}} \mathrm{e}^{\beta_p T},
	\end{align*}
	where the first inequality  is a consequence of the Minkowski's inequality for infinite sums and the last inequality is a consequence of Theorem \ref{thm:propagation_chaos}.
\end{proof}


\subsection{Proof of Theorem \ref{thm:quadbound}} \label{thm:proof_thm6}


In this section, we will consider that Assumption \ref{assump:additive+symmetric} is verified. Namely, the selection rates can be written as 
\(
V_\mu(x,y) = V_\mu^{\mathrm{d}}(x) + V_\mu^{\mathrm{b}}(y) + V_\mu^\mathrm{s}(x,y), 
\)
where $V_\mu^\mathrm{s}$ is symmetric.
Note that under \ref{assump:additive+symmetric}, equation \eqref{eq:preODE} becomes equivalent to
\begin{equation}\label{eq:EDO_appendix}
\partial_t \gamma_t(\phi) = \gamma_t \big( (Q + \Lambda) \phi - \gamma_t(\Lambda) \phi \big),
\end{equation}
where $\Lambda = V_\mu^{\mathrm{b}} - V_\mu^{\mathrm{d}}$.

Using \eqref{eq:EDO_appendix} we can simplify the expression of the martingale $\big(\mathcal{M}_t(\psi_\cdot)\big)_{t \ge 0}$ in Proposition \ref{prop:martingale} as follows
\begin{equation}
	\mathcal{M}_t(\psi_\cdot) = m(\eta_{t})(\psi_{t}) - m(\eta_0)(\psi_0) - \int_0^t m(\eta_{s}) \Big(\partial_s \psi_s + (Q + \Lambda) (\psi_{s}) - m(\eta_s)(\Lambda) \cdot \psi_s  \Big) \mathrm{d}s, \label{eq:reduction_martingale}
\end{equation}
for every bounded function $\psi$ on $E \times \mathbb{R}$, such that $\psi_\cdot(x)$ is continuously differentiable in $\mathbb{R_+}$, for every $x \in E$; and $\psi_t(\cdot) \in \mathcal{B}_b(E)$, for every $t \in \mathbb{R}$.
The previous expression is essential: for a suitable choice of the function $\psi$, we can control the integral part in the expression of $\mathcal{M}_t(\psi_\cdot)$.

Let us study the martingales obtained from Proposition \ref{prop:martingale} using as argument the function 
\[
    W_{t,T}(\phi) : t \in [0,T] \mapsto \frac{P_{T-t}^\Lambda (\phi)}{ \mu_t \left( P_{T-t}^\Lambda(\pmb{1}) \right)},
\]
i.e. where $W_{t,T}$ defined as in \eqref{eq:defWtT_intro} and $\phi \in \mathcal{B}_b(E)$ and $t \in [0,T]$.
Note that $W_{t,T}$ verifies the propagation equation $\mu_T(\phi) = \mu_t (W_{t,T}(\phi))$.
Besides,
\begin{align*}
	\partial_t \Big(\mu_t \big( P_{T-t}^\Lambda(\pmb{1}) \big) \Big) &= \partial_t \left( \frac{\mu_0 P_{T}^\Lambda (1) }{\mu_0 P_t^\Lambda(\pmb{1})} \right)
	= - \frac{\mu_0 P_T^\Lambda(\pmb{1})}{\mu_0 P_t^\Lambda(\pmb{1})^2} \mu_0 P_t^\Lambda(\Lambda)
	= - \mu_t\big(P_{T-t}^\Lambda(\pmb{1})\big) \mu_t(\Lambda).
\end{align*}
Thus,
\begin{align*}
	\partial_t W_{t,T}(\phi) &= - (Q + \Lambda) W_{t,T}(\phi) - \frac{\partial_t \big( \mu_t(P_{T-t}^\Lambda(\pmb{1})) \big) P_{T-t}^\Lambda(\phi) }{\mu_t P^\Lambda_{T-t}(1)^2} \\
	&= - (Q + \Lambda) W_{t,T}(\phi) + \mu_t(\Lambda) W_{t,T}(\phi).
\end{align*}
Hence, $W_{t,T}(\phi)$ is the solution of the Cauchy problem $\partial_s \psi_s = -\big( (Q + \Lambda ) - \mu_t(\Lambda)  \big) \psi_s$, with condition $\psi_T = \phi$.
Let us denote $\psi_{s,T} := W_{t,T}(\phi)$, for any $\phi \in \mathcal{B}_b(E)$.
Note that,
\[
	\partial_t \left( \psi_{t,T} \right) = - \Big( Q + \Lambda  - \mu_t(\Lambda)  \Big) \psi_{t,T} \; \text{ and } \;
	\partial_t \left( \psi_{t,T}^{2}\right) = -2 \psi_{t,T} \cdot \Big( \big( Q + \Lambda  - \mu_t(\Lambda) \big) \psi_{t,T}  \Big).
\]
We are in position to define the martingales $\big( \mathcal{M}_t(\psi_{\cdot, T}) \big)_{t \in [0,T]}$ and $\big( \mathcal{M}_t(\psi_{\cdot, T}^2) \big)_{t \in [0,T]}$, as stated in Proposition \ref{prop:martingale}, which under Assumption \ref{assump:additive+symmetric} can be written as follows
\begin{align}
	\mathcal{M}_t \left( \psi_{\cdot,T} \right) &=
	m(\eta_t)\left( \psi_{t,T} \right)
	- m(\eta_0)\left( \psi_{0,T} \right) 
	- \int_0^t m(\eta_s) \left( \psi_{s,T} \right) \left[ m(\eta_s) (\Lambda) - \mu_s(\Lambda) \right]\mathrm{d}s, \label{eq:martingale-linear} \\
	\mathcal{M}_t \left( \psi_{\cdot,T}^{2} \right)
	&= \, m(\eta_t)\left( \psi_{t,T}^{2} \right) -  m(\eta_0)\left( \psi_{0,T}^{2} \right) - 2  \int_0^t m(\eta_s) \left( \psi_{s,T}^2 \right) \Big[ \mu_s(\Lambda) - m(\eta_s)(\Lambda)  \Big] \mathrm{d}s - \Psi_t, \label{eq:martingale-carre}
\end{align}
where
\[
\Psi_t := \int_0^t m(\eta_s) \Big( \big(Q + \Lambda - m(\eta_s)(\Lambda) \big) (\psi_{s,T}^{2}) - 2 \psi_{s,T} \cdot \big( Q + \Lambda - m(\eta_s)(\Lambda) \big) (\psi_{s,T}) \Big) \mathrm{d}s.
\]

Furthermore, note that
\begin{equation}\label{eq:simply_calculus}
	\mu\big( \phi \cdot \widetilde{Q}_\mu \phi \big) = \mu\big( \phi (Q + \Lambda - \mu(\Lambda)) \phi \big) + \mu(\phi) \mu(\mathcal{V}_\mu \phi) - \mu(\phi^2 V^\mathrm{b}_\mu) - \mu(\phi^2) \mu( V^\mathrm{d}_\mu),
\end{equation}
where $\mathcal{V}_\mu := V_\mu^\mathrm{d} + V_\mu^\mathrm{b} \in \mathcal{B}_b(E)$, for every $\mu \in \mathcal{M}_1(E)$.

Thus, the predictable quadratic variation of $\big( \mathcal{M}_t(\psi_{\cdot, T}) \big)_{t \in [0,T]}$ satisfies
\begin{align*}
	N \Big\langle \mathcal{M}\Big( \psi_{\cdot,T} \big) \Big\rangle_t &:= \int_0^t m(\eta_s) \Big( \Gamma_{Q_{m(\eta_s)}}(\psi_{s,T}) \Big) \mathrm{d}s \\ 
	&= \int_0^t m(\eta_s) \left( 
	\widetilde{Q}_{m(\eta_s)} \big( \psi_{s,T}^2 \big) - 2 \psi_{s,T} \cdot \widetilde{Q}_{m(\eta_s)} \psi_{s,T} 
	\right) \mathrm{d} s + \int_0^t S_{m(\eta_s)} \left(\psi_{s,T} \right) \mathrm{d}s,
\end{align*}
where $S_\mu$ is defined as in \eqref{eq:defSmu}, for every $\mu \in \mathcal{M}_1(E)$.
Now, using \eqref{eq:simply_calculus} we obtain
\begin{align*}
	N \Big\langle \mathcal{M}\Big( \psi_{\cdot,T} \big) \Big\rangle_t = \Psi_t & - 2 \int_0^t m(\eta_s)(\psi_{s,T}) m(\eta_s)\big(\mathcal{V}_{m(\eta_s)} \psi_{s,T}\big) \mathrm{d}s + 2 \int_0^t  m(\eta_s) \big(\psi_{s,T}^2 V^\mathrm{b}_{m(\eta_s)} \big)\mathrm{d}s \\
	&+ 2 \int_0^t  m(\eta_s)(\psi_{s,T}^2) m(\eta_s)\left( V^\mathrm{d}_{m(\eta_s)} \right)\mathrm{d}s + \int_0^t S_{m(\eta_s)} \left(\psi_{s,T} \right) \mathrm{d}s.
\end{align*}
Then, using \eqref{eq:martingale-carre} we can substitute the value of $\Psi_t$ into this last expression and get
\begin{align}
	N \Big\langle \mathcal{M}\big( \psi_{\cdot,T} \big) \Big\rangle_t = &- \mathcal{M}_t(\psi_{\cdot, T}^2) + m(\eta_t)\left( \psi_{t,T}^{2} \right) -  m(\eta_0)\left( \psi_{0,T}^{2} \right)   + 2\int_0^t  m(\eta_s)(\psi_{s,T}^2) m(\eta_s)\left( V^\mathrm{d}_{m(\eta_s)} \right)\mathrm{d}s \nonumber \\ 
	&+ 2 \int_0^t  m(\eta_s) \left(\psi_{s,T}^2 V^\mathrm{b}_{m(\eta_s)} \right)\mathrm{d}s  +  \int_0^t S_{m(\eta_s)} \left(\psi_{s,T} \right) \mathrm{d}s + R_t, \label{eq:predictable-quad-variation_decomposition}
\end{align}
where
\begin{align}
	R_t := &- 2  \int_0^t m(\eta_s) \left( \psi_{s,T}^2 \right) \Big[ \mu_s(\Lambda) - m(\eta_s)(\Lambda)  \Big] \mathrm{d}s - 2 \int_0^t m(\eta_s)(\psi_{s,T}) m(\eta_s)\big(\mathcal{V}_{m(\eta_s)} \psi_{s,T}\big) \mathrm{d}s. \label{def:R_t}
\end{align}

The key component in the proof of Theorem \ref{thm:quadbound} is a central limit theorem for the martingale $\big(\mathcal{M}_t(\psi_{\cdot, T})\big)_{t \in [0,T]}$.
Let us first introduce an auxiliary result.

Consider the process $\big( \widetilde{\mathcal{M}}_t(W_{\cdot, T}(\bar{\phi}_T)) \big)_{t \in [0,T]}$ defined as
\[
\widetilde{\mathcal{M}}_t(W_{\cdot, T}(\bar{\phi}_T)) := \sqrt{N} m(\eta_0) \big( W_{0,T}(\bar{\phi}_T) \big) + \sqrt{N} \mathcal{M}_t(W_{\cdot, T}(\bar{\phi}_T)),
\]
for $t \in [0,T]$.
Then, $\big( \widetilde{\mathcal{M}}_t(W_{\cdot, T}(\bar{\phi}_T)) \big)_{t \in [0,T]}$ is a martingale, with initial value $$\widetilde{\mathcal{M}}_0(W_{\cdot, T}(\bar{\phi}_T)) = \sqrt{N} m(\eta_0) \big( W_{0,T}(\bar{\phi}_N) \big).$$

\begin{proposition}[Central limit theorem] \label{prop:CLT_martingale}
	The martingale $\big( \widetilde{\mathcal{M}}_t(W_{\cdot, T}(\bar{\phi}_T)) \big)_{t \in [0,T]}$ converges in law when $N \to \infty$ towards a Gaussian martingale whose variance at time $t \in [0,T]$ is $\sigma_t^2(\phi)$, defined as
	\[
	\sigma^2_t(\phi) := \mu_t\big( \psi_{t,T}^2 \big) + 2\int_0^t  \mu_s (\psi_{s,T}^2) \mu_s \left( V^\mathrm{d}_{\mu_s} \right)\mathrm{d}s + 2 \int_0^t  \mu_s \big(\psi_{s,T}^2 V^\mathrm{b}_{\mu_s} \big)\mathrm{d}s  +  \int_0^t S_{\mu_s} \left(\psi_{s,T} \right) \mathrm{d}s,
	\]
	and $\psi_{t,T} = W_{t,T}(\bar{\phi}_T)$.
\end{proposition}

\begin{proof}
	Using Theorem 3.11 in \cite[\S\ 8]{MR959133}, and arguing as in the proofs of Proposition 3.31 in \cite{DelMoralMiclo2000} and Proposition 3.7 in \cite{MR1956078}, we only need to check that the result holds for the initial value
	$\widetilde{\mathcal{M}}_0(W_{\cdot, T}(\bar{\phi}_T)) = \sqrt{N} m(\eta_0) \big( W_{0,T}(\bar{\phi}_T) \big)$ and that $N \langle \mathcal{M}(\psi_{\cdot, T}) \rangle$ converges in probability to a continuous function, when $N$ goes to infinity.
	The first point is in fact Assumption \ref{assump:initial_condition_as_normality}.
	Furthermore, 
	Theorem  \ref{thm:uniformLp} implies, by a Borel\,--\,Cantelli argument, the convergence
	\(
	m(\eta_s) \xrightarrow{ \mathrm{a.s.}} \mu_s,
	\)
	when $N \rightarrow \infty$, for all $s \ge 0$, see Remark \ref{rmk:almostsure_conv}.
	Now, using the Cauchy\,--\,Schwarz inequality and Theorem \ref{thm:uniformLp} (see also equation \eqref{eq:boundI2p}), we easily prove that $R_t$, defined by \eqref{def:R_t}, converges to $0$ in probability and that $N \langle \mathcal{M}(\psi_{\cdot, T}) \rangle$ converges to the continuous function $\sigma_\cdot^2(\phi) - \sigma_0^2(\phi)$ in probability, when $N \to \infty$, which concludes the proof.
\end{proof}	

\begin{proof}[Proof of Theorem \ref{thm:quadbound}]
	
	As a consequence of Proposition \ref{prop:CLT_martingale} and \eqref{eq:martingale-linear} we have that
	\[
	\widetilde{\mathcal{M}}_T\big( W_{\cdot, T}(\bar{\phi}_T) \big) = \sqrt{N} m(\eta_T)(\phi) - \mu_T(\phi) - \sqrt{N} \int_0^t m(\eta_s) \left( \psi_{s,T} \right) \left[ m(\eta_s) (\Lambda) - \mu_s(\Lambda) \right]\mathrm{d}s
	\]
	converges to a Gaussian random variable of variance $\sigma_T^2(\phi)$, when $N \to \infty$.
	Thus, the first part of Theorem \ref{thm:quadbound} comes from the fact that 
	\[
	\sqrt{N} \int_0^t m(\eta_s) \left( \psi_{s,T} \right) \left[ m(\eta_s) (\Lambda) - \mu_s(\Lambda) \right]\mathrm{d}s
	\]
	converges to $0$ almost surely when $N \to \infty$.
	Indeed, this is a consequence of the Cauchy\,--\,Schwarz inequality and Theorem \ref{thm:uniformLp} (see also \eqref{eq:boundI2p}).
	Thus, $m(\eta_T)(\phi) - \mu_T(\phi)$ converges in law to a centred Gaussian law with variance
	\begin{align*}
		\sigma_T^2(\phi) = & \mu_T \Big( \big( \phi - \mu_T (\phi)\big)^2 \Big) +  \int_0^T S_{\mu_s} \big( W_{s,T}(\bar{\phi}_T) \big) \mathrm{d}s  + 2   \int_0^T \mu_s \left( W_{s,T}(\bar{\phi}_T)^2 V_{\mu_s}^{\mathrm{b}} \right) +\mu_s \big( W_{s,T}(\bar{\phi}_T)^2\big) \mu_s \big( V_{\mu_s}^\mathrm{d} \big) \mathrm{d}s. 
	\end{align*}
	Consider now the change of variables $u = T - s$ in the last integral of the previous expression, and then take limit when $T \rightarrow \infty$.
	The final result comes due the following convergences:
	\begin{align*}
		\mu_{T-s} &\xrightarrow[\enskip T \to \infty \enskip]{} \mu_{\infty}, \\
		\bar{\phi}_T = \phi - \mu_T(\phi) &\xrightarrow[\enskip T \to \infty \enskip]{} \phi - \mu_{\infty}(\phi), \\ 
		W_{T-s,T} (\bar{\phi}_T) &\xrightarrow[\enskip T \to \infty \enskip]{} \frac{P_s^\Lambda (\bar{\phi}_\infty)}{ \mu_\infty P_s^\Lambda(\pmb{1}) } = \mathrm{e}^{- \lambda s} P_s^{\Lambda} (\bar{\phi}_\infty), 
	\end{align*}
	where the last identity is a consequence of the identity $\mu_\infty(\Lambda) = \lambda$ and 
\begin{equation} \label{eq:identity_Pt1}
	\mu_t \left( P_{T-t}^\Lambda(\pmb{1}) \right) = \exp \left\{ \int_t^T \mu_s(\Lambda) \mathrm{d} s \right\}.
\end{equation}
Indeed, \eqref{eq:identity_Pt1} is a consequence of the next two identities
\begin{align*}
	\frac{\mathrm{d}}{\mathrm{d} t} \ln \Big(\mu_0 P_t^\Lambda (\pmb{1}) \Big) = \mu_t(\Lambda)  \;\;\; \text{ and } \;\;\; \mu_t \left(P_{T-t}^\Lambda(\pmb{1}) \right) = \frac{\mu_0 \left( P_{T}^\Lambda(\pmb{1}) \right)}{\mu_0 \left(P_{t}^\Lambda(\pmb{1})\right)}.
\end{align*}
	
\end{proof}

\section*{Acknowledgements}

The authors would like to thank Djalil Chafa\"{i}, Pierre Del Moral, Simona Grusea, Vlada Limic and Denis Villemonais for their careful reading and helpful comments.
The authors would like to extend their gratitude to the anonymous reviewers whose comments and suggestions greatly improved the quality of this manuscript.
The work of the first author was partially supported by the ANR NOLO ANR-20-CE40-0015. 
The work of the second author was partially supported by the ITI IRMIA++.



\providecommand{\noopsort}[1]{}


\newpage

\appendix
\renewcommand{\thesection}{\Alph{section}}

\section{Proof of Lemma \ref{lemma:control_initial_condition}}\label{sec:appendix}

Let us first prove the following result, which has an independent interest.

\begin{lemma}[$\mathbb{L}^p$ norm bound for sum of i.i.d.\ centred r.v.]\label{lemma:bounding_norm_p_sum_va}
	Let us consider $Y_1,Y_2, \dots$ a sequence of independent identically distributed random variables with zero-mean and finite second moment, such that $\mathbb{E}[|Y_1|^p] < \infty$, for a given $p \ge 1$.
	Then, there exists a universal constant $C_p$ such that
	\[
	\left( \mathbb{E}\left[ \left| \frac{1}{N} \sum_{i = 1}^N Y_i \right|^p \right]  \right)^{1/p} \le \frac{C_p}{\sqrt{N}}.
	\]
\end{lemma}

\begin{proof}
	
	First note that for $p \le 2$ we get the following result as a consequence of Jensen inequality for concave functions:
	\begin{align*}
		\mathbb{E}\left[ \left| \frac{1}{N} \sum_{i = 1}^N Y_i \right|^p \right] &= \mathbb{E}\left[ \left( \left( \frac{1}{N} \sum_{i = 1}^N Y_i  \right)^2 \right)^{p/2} \right] \le \left( \mathbb{E}\left[  \left( \frac{1}{N} \sum_{i = 1}^N Y_i \right)^2  \right]\right)^{p/2} = \left( \frac{\mathbb{E}[Y^2]}{N} \right)^{p/2}.
	\end{align*}
	
	For $p > 2$, the proof follows from
	Marcinkiewicz\,--\,Zygmund inequality, which is a consequence of the BDG inequality for discrete-time martingales.
	Indeed, the Marcinkiewicz–Zygmund inequality (cf.\ \cite{MR1841623}) ensures us that
	\begin{equation}\label{eq:MZineq}
		\mathbb{E}\left[ \left| \sum_{i=1}^N Y_i\right|^p \right] \le \frac{K_p}{N^{p/2}} \mathbb{E}\left[ |Y_1|^p \right].	
	\end{equation}
	
	Thus,
	\[
	\mathbb{E}\left[  \left| \frac{1}{N} \sum_{i=0}^N Y_i \right|^p \right] \le \frac{C_p}{N^{p/2}}, \text{ where } C_p = \left\{
	\begin{array}{ccc}
		(\mathbb{E}[Y_1^2])^{p/2} & \text{ if } & p \le 2\\
		K_p \mathbb{E}[|Y_1|^p] & \text{ if } & p > 2.
	\end{array}
	\right.
	\]
	
\end{proof}

\begin{remark}[Qualitative results for the Marcinkiewicz–Zygmund constant $K_p$]
	See the work of Ren and Liang \cite{MR1841623} for a qualitative study of the constant $K_p$ in inequality \eqref{eq:MZineq}.
	They show that $(K_p)^{1/p}$ grows like $\sqrt{p}$, when $p \rightarrow \infty$, and give the estimate $K_p \le (3 \sqrt{2})^p p^{p/2}$. 
\end{remark}

\section{Auxiliary results for proving Theorem \ref{thm:propagation_chaos} }
\begin{proof}[Proof of Lemma \ref{lemma:control_initial_condition}]
	Note that
	\(
	m(\eta_0)(\phi) = \frac{1}{N} \sum_{i = 1}^N \phi\big(\xi_0^{(i)}\big),
	\)
	where $\xi_0^{(i)}$, for $i = 1,\dots,N$ are independent random variables.
	Moreover, $\phi\big(\xi_0^{(i)}\big)$ has mean $\mu_0(\phi)$, for all $i = 1,\dots,N$.
	Thus,
	\[
	m(\eta_0)(\phi) - \mu_0(\phi) = \sum_{i = 1}^N \frac{\phi\big(\xi_0^{(i)}\big) - \mu_0(\phi)}{N},
	\]
	can be written as a sum of $N$ zero-mean random variables. 
	The result comes from Lemma \ref{lemma:bounding_norm_p_sum_va}.
\end{proof}

\subsection{Proof of Lemma \ref{lemma:1} } \label{app:proof_lemma:1}

The first identity is simply a consequence of \eqref{eq:gen_mut} and \eqref{eq:gen_sel}, and the fact that
\begin{equation}\label{eq:removing_symmetric}
	\mu(Q_\mu \phi) = \mu\big(\widetilde{Q}_\mu \phi\big).
\end{equation}
Now, to prove the second identity, note that
\begin{align*}
	\Gamma_{\mathcal{Q}} \Big(m(\cdot)(\phi) \Big)(\eta) &= \sum_{x \in E} \eta(x) \sum_{y \in E} \left( Q_{x,y} + V_{m(\eta)}(x,y) \frac{\eta(y)}{N} \right) [m(\eta - \mathbf{e}_x + \mathbf{e}_y)(\phi) - m(\eta)(\phi)]^2 \\
	&= \frac{1}{N} \sum_{x \in E} \frac{\eta(x)}{N} \sum_{y \in E} \left( Q_{x,y} + V_{m(\eta)}(x,y) \frac{\eta(y)}{N} \right) [\phi(y) - \phi(x)]^2\\
	&= \frac{1}{N} \sum_{x \in E} \left( \sum_{y \in E} \big(Q_{x,y} + V_{m(\eta)}(x,y) m_ y(\eta) \big) [\phi(y) - \phi(x)]^2 \right) m_{x}(\eta)\\ 
	&= \frac{1}{N} m(\eta) \left( \Gamma_{Q_{m(\eta)}} (\phi) \right). \;\; \qedsymbol
\end{align*}

\subsection{Proof of Proposition \ref{prop:martingale}} \label{app:proof_prop:martingale}
	The usual martingale problem associated to $(\eta_t)_{t \ge 0}$ implies that for every function $\phi$ on $E$, the process
	\begin{align*}
		t \mapsto& m(\eta_t) (\phi) - m(\eta_0) (\phi) - \int_0^t \mathcal{Q} (m(\eta_s) (\phi)) \mathrm{d}s \\
		&=
		m(\eta_t) (\phi) - m(\eta_0) (\phi) - \int_0^t m(\eta_s) \big(\widetilde{Q}_{m(\eta_s)} (\phi)\big) \mathrm{d}s 
	\end{align*}
	is a local martingale.
	Note that the equality is due to the first identity in Lemma \ref{lemma:1}.
	Then, for a function $\psi$ on $E \times \mathbb{R}_+$, continuously differentiable in $\mathbb{R}_+$, the It\^{o} formula implies that
	\( 
	\left(\mathcal{M}_t(\psi_\cdot)\right)_{t \ge 0}
	\)
	is a local martingale, as desired.
	
	The predictable quadratic variation is obtained using the identity
	\[
	\langle \mathcal{M}(\psi) \rangle_t = \int_0^t \Gamma_{\mathcal{Q}}\big(m(\eta_s)(\psi_s) \big) \mathrm{d} s,
	\]
	and the final result comes from the second identity in Lemma \ref{lemma:1}.
	
	Furthermore, the bound for the jump is due to the fact that each jump only concerns one particle that jumps from one position to another. \qedsymbol
	
\subsection{Proof of Lemma \ref{lemma:bound_quadratic_variation:additive}}
\label{app:proof_lemma:bound_quadratic_variation:additive}

The predictable quadratic variation of the martingale $\Big( \mathcal{M}_t \left( P(\cdot, T)\big(\phi\big) \right) \Big)_{t \in [0,T]}$ satisfies
\begin{align*}
	N \Big\langle \mathcal{M}\Big(P(\cdot, T)\big( \phi \big) \Big) \Big\rangle_t &= \int_0^t m(\eta_s) \left( \Gamma_{Q_{m(\eta_s)}} \left(P(s,T)\big( \phi \big) \right) \right) \mathrm{d}s\\
	&= \int_0^t m(\eta_s) \left( \widetilde{Q}_{m(\eta_s)} \left( P(s,T)\big( \phi \big)^{2} \right) - 2 P(s,T)\big( \phi \big) \cdot Q_{m(\eta_s)} \Big( P(s,T)\big( \phi \big) \Big) \right) \mathrm{d}s,
\end{align*}
where the second equality holds because of the definition of {carr\'e-du-champ} operator and \eqref{eq:removing_symmetric}.

Thus, using \eqref{eq:martingalePtT2} we get
\begin{align*}
	N \Big\langle \mathcal{M}&\Big(P(\cdot, T)\big( \phi \big) \Big) \Big\rangle_t 
	= - \mathcal{M}_t \Big( P(\cdot, T)\big( \phi \big)^{2} \Big) - m(\eta_t) \Big( P(t,T)\big( \phi \big)^{2} \Big) + m(\eta_0) \Big( P(0,T)\big( \phi \big)^{2} \Big)\\
	&+  2 \int_0^t m(\eta_s)\Big( P(s,T) \big( \phi \big) \cdot \Big[ \Big( \widetilde{Q}_{\mu_s} - Q_{m(\eta_s)} \Big) \Big( P(s,T)\big( \phi \big) \Big) \Big] \Big) \mathrm{d}s.
\end{align*}

Now, because of \eqref{eq:bound_integral_WtT} and the boundedness conditions on $V_\mu$ in Assumption \ref{assump:gral_slection rate} we can ensure the existence of a constant $C > 0$ such that
\[
N \Big\langle \mathcal{M}\Big(P(\cdot, T)\big( \phi \big)\Big) \Big\rangle_t \le C( t + 1) - \mathcal{M}_t \left( P(\cdot, T)\big( \phi \big)^{2} \right). \;\; \qedsymbol
\]

\subsection{Proof of Lemma \ref{thm:control_quadratic_variation:additive}}
\label{app:proof_thm:control_quadratic_variation:additive}

First, by localisation, we can suppose that the martingales are bounded.
Now, we will prove the inequalities for $p = 2^q$, and then using the Jensen inequality we will extend the result for all $p \ge 1$.
The result for $p \in (0,1)$ is simply a consequence of the result for $p = 1$ and the Jensen inequality for concave functions.

We want to prove the following inequalities:
\begin{align*}
	\mathbb{E} \left[ \left( \Big\langle \mathcal{M}\Big( P(\cdot, T)\big( \phi \big) \Big) \Big\rangle_t \right)^{2^q} \right] &\le \frac{C (t + 1)^{2^q}}{N^{2^q}},\;\;\;    \mathbb{E} \left[ \left( \left[ \mathcal{M}\Big( P(\cdot, T)\big( \phi \big) \Big) \right]_t \right)^{2^q} \right] \le \frac{C (t + 1)^{2^q}}{N^{2^q}}.
\end{align*}
For $q = 0$, the first inequality is a consequence of Lemma \ref{lemma:bound_quadratic_variation:additive} and the second one is due to the fact that $\left( \left[ \mathcal{M}\Big( P(\cdot, T)\big( \phi \big) \Big) \right] -  \Big\langle \mathcal{M}\Big( P(\cdot, T)\big( \phi \big) \Big) \Big\rangle \right)_{t \in [0,T]}$ is a local martingale.

We will prove the previous inequalities by induction.
Let us assume they are true for $q$ and lower.    
Thus, by Lemma \ref{lemma:bound_quadratic_variation:additive} and the Minkowski inequality, there exists a $K > 0$ such that
\begin{align*}
	I_p &:= \mathbb{E} \left[ \Big(N \Big\langle \mathcal{M}\big( P(\cdot, T) ( \phi ) \big) \Big\rangle_t \Big)^{p} \right] \\
	&\le \mathbb{E}\left[  \left( K ( t + 1) + \left|\mathcal{M}_t \Big(P(\cdot, T)\big( \phi \big)^{2}\Big) \right| \right)^p \right] \\
	&\le \left( K( t + 1 ) + \left( \mathbb{E}  \left[ \left| \mathcal{M}_t \left( P(\cdot, T)\big( \phi \big)^{2} \right) \right|^{p} \right] \right)^{1/p}  \right)^{p},
\end{align*}
for all $p \ge 1$.
Using now the BDG inequality, we get
\begin{align*}
	I_{2^{q+1}} &\le \left( K (t+1) + \kappa \left( \mathbb{E} \left[ \left( \left[ \mathcal{M}\Big( P(\cdot, T)\big( \phi \big)^{2} \Big) \right]_t \right)^{2^q} \right]\right) ^{1/2^{q+1}} \right)^{2^{q+1}} \\
	&\le \left( K (t+1) + \kappa \sqrt{\frac{ t + 1}{N} } \right)^{2^{q+1}} \le C' ( t + 1)^{2^{q+1}},
\end{align*}
where the second inequality holds by the induction hypothesis and the last one due to $N \ge 1$ and $t+1 \ge 1$.
Now, the martingale $\left(\mathcal{M}_t\big(P(\cdot, T)(\phi) \big)\right)_{t \in [0,T]}$ has jumps verifying
\[
a \le 2 \frac{\left\|P(\cdot, T)\big( \phi \big) \right\|}{N} \le \frac{2}{N}.
\]

Thus, using Lemma \ref{lemma:BDGineq} we get
\begin{align*}
	\mathbb{E}\left[ \Big( \big[ \mathcal{M}\big(P(\cdot, T)(\phi) \big) \big]_t \Big)^{2^{q+1}} \right] &\le C'' \sum_{k = 0}^{q+1} \frac{\mathbb{E} \left[ \Big( \langle \mathcal{M} \big( P(\cdot, T)(\phi) \big) \rangle_t \Big)^{2^k} \right]}{N^{2^{q+2} - 2^{k+1}}} \le C'' \sum_{k = 0}^{q+1} \frac{(t+1)^{2^k}}{N^{2^{q+2} - 2^{k+1} + 2^k}} \\
	&= \frac{C''}{N^{2^{q+2}}} \sum_{k = 0}^{q+1} \big[N(t+1)\big]^{2^k} = \frac{C'' (q+1)}{N^{2^{q+2}}} N^{2^{q+1}} (t+1)^{2^{q+1}}\\
	&\le C \frac{ (t+1)^{2^{q+1}} }{ N^{2^{q+1}} }.
\end{align*}
This concludes the proof for $p = 2^q$.

Now, for arbitrary $p$, there exists $q$ such that $p \le 2^{q}$.
Thus, using the Jensen inequality (for the concave function $x \mapsto x^{p/2^q}$) we get
\begin{align*}
	\mathbb{E} \left[ \left( \Big\langle \mathcal{M}\Big( P(\cdot, T)\big( \phi \big) \Big) \Big\rangle_t \right)^{p} \right] &\le \left( \mathbb{E} \left[  \left( \Big\langle \mathcal{M}\Big( P(\cdot, T)\big( \phi \big) \Big) \Big\rangle_t \right)^{2^{q}}  \right] \right)^{{p}/{2^q}}\\
	&\le \left( \frac{C (t + 1)^{2^q}}{N^{2^q}} \right)^{p/2^q} \le C^{p/2^q} \frac{(t+1)^{p}}{N^p}.
\end{align*}
The result for $\mathbb{E}\left[ \left( \big[ \mathcal{M}\big(P(\cdot, T) (\phi) \big) \big]_t \right)^{p} \right]$ is analogously obtained. \qedsymbol


\section{Proof of Theorems \ref{thm:uniformLp} and its corollaries } \label{sec:proof_thm2-3}

Obtaining a uniform in time bound as the one provided by Theorem \ref{thm:uniformLp} is a hard problem and this kind of results are uncommon in the literature.
Del Moral and Guionnet in \cite[Thm.\ 3.1]{DelMoralGuionnet2001} have proved a similar result for a resembling but discrete-time model, where the potential function $\Lambda$ is assumed uniformly bounded and also bounded away from zero.
Moreover, their upper bound for the speed of convergence is of order $1/N^\alpha$, with $\alpha < 1/2$.
The same order of convergence $1/N^\alpha$ for a similar model is obtained in Theorem 2.11 in \cite{DelMoralMiclo2000} for the control of the convergence in $\mathcal{F}$-norms, and not only for a single test function.
Also, Theorem 14.3.7 in \cite[\S\ 14.3.3]{zbMATH06190205} provides uniform estimates for time-discretization models, see \cite[\S\ 12.2.3.1]{zbMATH06190205}.
However, those estimates depend on the time mesh and are not bounded when the size of the time steps tend to zero.
Also in the discrete-time settings, see Theorem 5.8 and Corollary 5.12 in \cite{zbMATH02072697}, where a uniform in time bound with the optimal order $1/\sqrt{N}$ is proved.
Rousset \cite[Thm.\ 4.1]{MR2262944} has proved a uniform in time bound in $\mathbb{L}^p$ with the same speed of convergence as our result.
However, the model studied by Rousset is in continuous state space and the diffusion process driving the mutation process is assumed reversible.
Similarly, Angeli et al.\ \cite[Thm.\ 3.2]{Angeli2021} obtained an equivalent result for jump processes on locally compact spaces  in the context of cloning algorithms and for $p \ge 2$.
See also Theorem 5.10 and Corollary 5.12 in \cite{MR4193898}, for a related result when $p=2$.

Our model is different, since we consider the case where the state space is countable and discrete, not necessarily finite
and  in Assumption \ref{assump:additive+symmetric} we allow the selection rates to depend on the empirical probability measure induced by the particle system, in the same spirit of \cite{MR2262944}.
Nonetheless, our methods are similar to those of Rousset \cite{MR2262944} and Angeli et al.\ \cite{Angeli2021} (see also \cite[\S\  3.3.1]{DelMoralMiclo2000}): it consists in finding a martingale indexed by the interval $[0, T]$, whose terminal value at time $T$ is precisely
$m(\eta_T)(\phi) - \mu_T(\phi)$ plus a term whose $\mathbb{L}^p$ norm can be controlled, for any $\phi \in \mathcal{B}_b(E)$.
Thereafter, the final result comes by a control of the quadratic variation of the martingale and an induction principle.

In the rest of this section, we will consider that Assumption \ref{assump:additive+symmetric} is verified. Namely, the selection rates can be written as follows 
\(
V_\mu(x,y) = V_\mu^{\mathrm{d}}(x) + V_\mu^{\mathrm{b}}(y) + V_\mu^\mathrm{s}(x,y), 
\)
where $V_\mu^\mathrm{s}$ is symmetric.

Let us denote by $\bar{m}(\eta_t)$ the mean empirical probability measure induced by $\eta_t$, which is defined as
\[
\bar{m}(\eta_t) := \sum_{x \in E} \mathbb{E}\left[ \frac{ \eta_t (x)}{N} \right] \delta_x \in \mathcal{M}_1(E).	
\]

When $V_\mu^{\mathrm{d}}$ and $V_\mu^{\mathrm{b}}$ are nulls, and thus $V_\mu$ is symmetric, we obtain the following result as an immediate consequence of \eqref{eq:reduction_martingale}.

\begin{corollary}[$V_\mu$ is symmetric]\label{cor:Vsymmetric}
	Assume that Assumption \ref{assump:additive+symmetric} is verified in such a way that $V_{\mu} = V_{\mu}^\mathrm{s}$, and Assumptions \ref{assump:initial_condition} and \ref{assump:ergod_normalised} are also verified.
	Then the process
	\[
	\Big( m(\eta_t)\big( \mathrm{e}^{(T-t) Q} ( \phi ) \big)  - m(\eta_0)( \mathrm{e}^{T Q} \big(\phi) \big) \Big)_{t \in [0,T] }, 
	\]
	is a local martingale, for every $\phi \in \mathcal{B}_b(E)$. 
	In particular,
	\(
	\bar{m}(\eta_t) = \bar{m}(\eta_0) \mathrm{e}^{t Q},	
	\)
	for all $t \ge 0$.
\end{corollary}

\begin{proof}[Proof of Corollary \ref{cor:Vsymmetric}]
	Note that since the selection rates are symmetric, $\Lambda$ as defined in \eqref{eq:defW} is null.
	The proof simply follows as a consequence of \eqref{eq:reduction_martingale}, taking $\psi_t = \mathrm{e}^{(T-t) Q}(\phi)$, for all $t \in [0,T]$ and $\phi \in \mathcal{B}_b(E)$.
\end{proof}

Let us recall the operator $W_{t,T}$ defined in \eqref{eq:defWtT_intro} as
\(
W_{t,T}	: \phi \mapsto {P_{T-t}^\Lambda (\phi)}/{ \mu_t \Big( P_{T-t}^\Lambda(\pmb{1}) \Big)}.
\)
Recall that $\bar{\phi}_T := \phi - \mu_T(\phi)$.
We get
\(
m(\eta_T) \big(W_{T,T}( \bar{\phi}_T ) \big) = m(\eta_T)(\phi) - \mu_T(\phi),
\)
which is the difference we intend to control.
The following results establish a control of the uniform norm of $W_{t,T}$.

\begin{lemma}\label{lemma:controlling_propagator}
	The operator $(W_{t,T})_{t \in [0,T]}$ verifies the following properties:
	\begin{itemize}
		\item[a)] Given $p \ge 1$, for any test function $\phi \in \mathcal{B}_1(E)$, there exists $C> 0$ such that
		\begin{align*}
			\|W_{t,T} (\phi)\| \le C, \;\; \text{ and } \;\;
			\int_t^T \|W_{s,T}(\phi)\|^{p} \mathrm{d}s \le  C (T - t).  
		\end{align*}
		\item[b)] There exists a constant $\rho \in (0,1)$, such that
		\begin{align*}
			\big\|W_{t,T} \big(\bar{\phi}_T\big)\big\| \le C \rho^{ T - t}, \;\; \text{ and } \;\;
			\int_t^T \big\|W_{s,T}\big(\bar{\phi}_T\big)\big\|^{p} \mathrm{d}s \le C.  
		\end{align*}
	\end{itemize}
\end{lemma}

\begin{proof}[Proof of Lemma \ref{lemma:controlling_propagator}]
	
	The proof of this result is inspired by the proof of Lemma 5.1 in \cite{MR2262944}, but we do not make any direct assumption of the spectrum of $Q + \Lambda$.
	
	Note that
	\(
	\mu_t \big( P_{T-t}^\Lambda (\pmb{1}) \big) = {\mu_0 P_{T}^\Lambda (\pmb{1}) }/{\mu_0 P_t^\Lambda(\pmb{1})}.
	\)
	Moreover, using Lemma \ref{corollary:exp_ergodicity_nonnormalised}, the function $t \mapsto \mathrm{e}^{-\lambda t} \mu_0 \big( P^{\Lambda}_t(\pmb{1}) \big)$ is continuous and positive, going from $1$ to $\mu_0(h) > 0$.
	This proves part a).
	
	To prove part b) of the lemma, note that
	\[
	\mu_T(\phi) = \frac{\mu_t P_{T-t}^\Lambda(\phi)}{\mu_t P_{T-t}^\Lambda(\pmb{1})}
	\;\;\; \text{ and } \;\;\; 
	W_{t,T}(\mu_T(\phi)) = \mu_T(\phi) \frac{P_{T-t}^\Lambda(\pmb{1})}{\mu_t P_{T-t}^\Lambda(\pmb{1})},
	\]
	since $\mu_T(\phi)$ is constant.
	Thus,
	\begin{align*}
		\| W_{t,T}(\bar{\phi}_T) \| &= \left\| \frac{P_{T-t}^\Lambda (\phi)}{\mu_t P_{T-t}^\Lambda (\pmb{1})} - \mu_T(\phi) \frac{P_{T-t}^\Lambda(\pmb{1})}{\mu_t P_{T-t}^\Lambda(\pmb{1})} \right\| \\
		&= \left\| \frac{  \mu_t P_{T-t}^\Lambda(\pmb{1}) \cdot P_{T-t}^\Lambda(\phi) -  \mu_t P_{T-t}^\Lambda(\phi) \cdot P_{T-t}^\Lambda(\pmb{1})  }{
			\left( \mu_t P_{T-t}^\Lambda(\pmb{1}) \right)^2
		}  \right\| \le C \rho^{T-t},
	\end{align*}
	where the last inequality is a consequence of the fact that the function $t \mapsto \mathrm{e}^{- \lambda t} \mu_0 P^{\Lambda}_t(\pmb{1})$ is bounded away from zero, and the uniform convergence of $\mathrm{e}^{-\lambda t} P^\Lambda_t (\phi)$ towards $h \mu_{\infty}(\phi)$, when $t \to \infty$, claimed in Lemma \ref{corollary:exp_ergodicity_nonnormalised}.
\end{proof}

Let $\psi$ be a function on $E \times \mathbb{R}_+$ as in the statement of Proposition \ref{prop:martingale}.
We denote by $\big( \mathcal{M}^T_t (\psi_{\cdot, T}) \big)_{t \in [0,T]}$ the local martingale defined as
\(
	\mathcal{M}_t^T \left( \psi_{\cdot, T} \big(\phi\big) \right) := \mathcal{M}_T \left( \psi_{\cdot, T} \big(\phi\big) \right) - \mathcal{M}_t \left( \psi_{\cdot, T} \big(\phi\big) \right),
\)
for all $t \in [0,T]$.
We denote by $\big(\langle \mathcal{M} (\psi_\cdot) \rangle_t^T \big)_{t \in [0,T]}$, the predictable quadratic variation of the local martingale $\big( \mathcal{M}^T_t (\psi_{\cdot, T}) \big)_{t \in [0,T]}$.

Using \eqref{eq:predictable-quad-variation_decomposition} we can prove the next two results, analogously to Lemmas \ref{lemma:bound_quadratic_variation:additive} and \ref{thm:control_quadratic_variation:additive}, establishing a control on the predictable quadratic variation and the quadratic variation of the martingale  $\big( \mathcal{M}^T_t (\psi_{\cdot, T}) \big)_{t \in [0,T]}$, respectively.

\begin{lemma}[Control of the predictable quadratic variation]
	\label{lemma:bound_quadratic_variation}
	For every test function $\phi \in \mathcal{B}_1(E)$ and every $t \in [0,T]$ we have
	\begin{align*}
		N \Big\langle \mathcal{M}\left(W_{\cdot, T}\big( \phi \big) \right) \Big\rangle_t^T &\le C (T - t + 1) - \mathcal{M}_t^T \Big(W_{\cdot, T}\big( \phi \big)^{2}\Big), \;\; \text{ for all } \;\;  t \in [0,T],
	\end{align*}
	and for $\bar{\phi}_T = \phi - \mu_T(\phi)$ we have
	\[
	N \Big\langle \mathcal{M}\left(W_{\cdot, T}\left(\bar{\phi}_T\right) \right) \Big \rangle_t^T \le C - \mathcal{M}_t^T \left(W_{\cdot, T} \left( \bar{\phi}_T \right)^{2} \right), \;\; \text{ for all } \;\;  t \in [0,T].
	\]
	
\end{lemma}

\begin{lemma}[Control of the quadratic variation]\label{thm:control_quadratic_variation}
	For all $p > 0$ and every test function $\phi \in \mathcal{B}_1(E)$ there exists a positive $C_p$ (possibly depending on $p$) such that
	\[
	\mathbb{E} \left[ \left( \left[ \mathcal{M}\Big( W_{\cdot, T}\big( \phi \big) \Big) \right]_t^T \right)^p \right] \le \frac{C_p ( T - t + 1)^p}{N^p},
	\]
	and for a centred test function $\bar{\phi}_T = \phi - \mu_T(\phi)$:
	\[
	\mathbb{E} \left[ \left( \left[ \mathcal{M}\Big( W_{\cdot, T}\big( \bar{\phi}_T \big) \Big) \right]_t^T \right)^p \right] \le \frac{C_p}{N^p}.
	\]
\end{lemma}

The proofs of Lemmas \ref{lemma:bound_quadratic_variation} and \ref{thm:control_quadratic_variation} are analogous to those of Lemmas \ref{lemma:bound_quadratic_variation:additive} and \ref{thm:control_quadratic_variation:additive}, respectively.
They are obtained using Lemma \ref{lemma:controlling_propagator} instead of \eqref{eq:bound_integral_WtT}.
In particular, the second inequalities in both results are consequences of part b) of Lemma \ref{lemma:controlling_propagator}. 
See also the proofs of Lemma 5.3 and Theorem 5.4 in \cite{MR2262944}.
We skip the proofs of Lemmas \ref{lemma:bound_quadratic_variation} and \ref{thm:control_quadratic_variation} for the sake of brevity.

Let us define the nonlinear propagator associated to $(\mu_t)_{t \ge 0}$ as follows
\[
\Phi_{t,T}(\nu) := \frac{\nu P_{T-t}^\Lambda}{\nu P_{T-t}^\Lambda(\pmb{1})} \in \mathcal{M}_1(E).
\]
By the semigroup property, $\Phi_{t,T}$ satisfies the propagation equation $\mu_T = \Phi_{t,T}(\mu_t)$.
Using Assumption \ref{assump:ergod_normalised} we can ensure the existence of $\rho \in (0,1)$ such that
\(
\sup_{\nu \in \mathcal{M}_1(E)} \| \Phi_{t,T}(\nu) - \mu_{\infty} \|_{\mathrm{TV}} \le C \rho^{T-t}.
\)

Let us define
\[
I_p(N) = \sup_{T \ge 0} \sup_{\phi \in \mathcal{B}_1(E)} \mathbb{E}\left[ \big( m(\eta_T)(\phi) - \mu_T(\phi) \big)^p \right].
\]
Our goal is to prove that
\(
I_p(N) \le {C}/{N^{p/2}}.
\)
The method we use is similar to the one used by Rousset \cite{MR2262944} and Angeli et al.\ \cite{Angeli2021}. 
Broadly speaking, it consists in an induction principle.
First let us prove the following result providing the initial case of the induction.

\begin{lemma}[Initial case]\label{lemma:keylemma}
	There exists $\epsilon > 0$ independent of $p$, such that
	\[
	I_p(N) \le \frac{C}{N^{\epsilon p/2}}.
	\]
\end{lemma}

\begin{proof}
	Fix $T > 0$ and consider
	\begin{equation}\label{eq:obj_function}
		m(\eta_T)(\phi) - \mu_T(\phi) =  \underbrace{m(\eta_T)(\phi) - \Phi_{t,T}\big(m(\eta_t)\big)(\phi)}_{:= a(t)} + \underbrace{\Phi_{t,T}\big(m(\eta_t)\big)(\phi) - \mu_T(\phi)}_{:= b(t)}.
	\end{equation}
	The idea is to control $a(t)$ using the stochastic error between $t$ and $T$, and $b(t)$ using the limiting stability.
	Moreover, $b(0)$ is controlled by the error made by the initial condition.
	
	Let us first control the term $\mathbb{E}[|a(t)|^p]$. 
	Consider the finite variation process
	\[
	A_{t_1}^{t_2} := \exp\left\{ \int_{t_1}^{t_2} m(\eta_s)(\Lambda) - \mu_s(\Lambda) \mathrm{d}s \right\}.
	\]
	Then,
	\begin{equation}\label{eq:operatorA}
		\partial_s \Big(
		A_t^s m(\eta_s) \big( W_{s,t}(\phi) \big)\Big) = A_t^s \mathrm{d} \mathcal{M}_s \big( W_{\cdot, T}(\phi) \big).
	\end{equation}
	Indeed, the martingale problem in Proposition \ref{prop:martingale}, for the function $W_{t,T}(\phi)$ yields
	\[
	\mathrm{d} \Big( m(\eta_t)(W_{t,T}(\phi)) \Big) = \mathrm{d} \mathcal{M}_t(W_{\cdot, T}(\phi)) + \big(\mu_t(\Lambda) - m(\eta_t)\big) m(\eta_t) (W_{t,T}(\phi))\mathrm{d}t.
	\]
	Hence,
	\begin{align*}
		\partial_s \left( A_t^s m(\eta_s)\big( W_{s,T}(\phi) \big) \right) &= \partial_s \left( A_t^s \right) m(\eta_s)\big( W_{s,T}(\phi) \big) + A_t^s \mathrm{d} \left( m(\eta_s)\big( W_{s,T}(\phi) \big) \right),\\
		&= A_t^s \Big( m(\eta_s)(\Lambda) - \mu_s(\Lambda) \Big)  m(\eta_s)\big( W_{s,T}(\phi) \big) + A_t^s \mathrm{d} \left( m(\eta_s)\big( W_{s,T}(\phi) \big) \right)\\
		&= A_t^s \mathrm{d} \mathcal{M}_s \big( W_{\cdot, T}(\phi) \big),
	\end{align*}
	where the last expression is obtained using \eqref{eq:martingale-linear}.
	
	Now, integrating from $t$ to $T$ in \eqref{eq:operatorA} and dividing by $A_t^T$ we get
	\[
	m(\eta_T)(\phi) - \big(A_t^T\big)^{-1} m(\eta_t) \big( W_{t,T}(\phi) \big) = \big(A_t^T\big)^{-1} \int_t^T A_t^s \mathrm{d} \mathcal{M}_s \big( W_{\cdot, T}(\phi) \big).
	\]
	Note that
	\[
	\Phi_{t,T}(m(\eta_t))(\phi) = \frac{(A_t^T)^{-1} m(\eta_t) (W_{t,T}(\phi))}{(A_t^T)^{-1} m(\eta_t) (W_{t,T}(\pmb{1}))},
	\]
	for all $t \le T$.
	Thus,
	\begin{align*}
		a(t) &= m(\eta_t)(\phi) - \big(A_t^T\big)^{-1} m(\eta_t) \big( W_{t,T}(\phi) \big) - \Phi_{t,T}(m(\eta_t))(\phi) \left[ 1 - \big(A_t^T\big)^{-1}m(\eta_t)\big(W_{t,T}(\pmb{1})\big) \right]\\
		&= \big(A_t^T\big)^{-1} \int_t^T A_t^s \mathrm{d} \mathcal{M}_s \big( W_{\cdot, T}(\phi) \big) - \Phi_{t,T}(m(\eta_t))(\phi) (A_t^T)^{-1} \int_t^T A_t^s \mathrm{d}\mathcal{M}_s(W_{\cdot, T}(\pmb{1})).
	\end{align*}
	Thus, we obtain the upper bound
	\[
	|a(t)| \le (A_t^T)^{-1} \left( \left| \int_t^T A_t^s \mathrm{d}\mathcal{M}_s\Big(W_{\cdot, T}(\phi)\Big) \right| + \left| \int_t^T A_t^s \mathrm{d}\mathcal{M}_s\Big(W_{\cdot, T}(\pmb{1})\Big) \right|  \right).
	\]
	There exists a $K > 0$ such that
	\begin{align*}
		\mathbb{E}\left[ |a(t)|^p \right] &\le  K \mathrm{e}^{2 p \|\Lambda\| (T-t)} \sup_{\varphi \in \mathcal{B}_1(E)} \mathbb{E} \left[ \left| \int_t^T A_t^s \mathrm{d} \mathcal{M}_s\big( W_{s,t} (\varphi) \big) \right|^p \right]\\
		&\le K \mathrm{e}^{2 p \|\Lambda\| (T-t)} \sup_{\varphi \in \mathcal{B}_1(E)}  \mathbb{E} \left[ \left| \int_t^T (A_t^s)^2 \mathrm{d} \big[ \mathcal{M}\big( W_{\cdot,t} (\varphi) \big) \big]_s \right|^{p/2} \right],
	\end{align*}
	where the second inequality holds by the BDG inequality.
	Then, using Lemma \ref{thm:control_quadratic_variation} we get
	\begin{align*}
		\mathbb{E}\left[ |a(t)|^p \right] &\le K \mathrm{e}^{4 p \|\Lambda\| (T-t)} \sup_{\varphi \in \mathcal{B}_1(E)} \mathbb{E} \left[ \left| \big[ \mathcal{M}\big( W_{\cdot,t} (\varphi) \big) \big]_t^T \right|^{p/2} \right]\\
		&\le K \mathrm{e}^{4 p \|\Lambda\| (T-t)} \frac{(T-t+1)^{p/2}}{N^{p/2}} \le K \frac{\kappa^{p(T-t)}}{N^{p/2}},
	\end{align*}
	where $\kappa = \mathrm{e}^{4 \|\Lambda\| + 1/2} > 1$.
	
	Let us now control $\mathbb{E}[|b(t)|^p]$.
	As a consequence of Assumption \ref{assump:ergod_normalised} there exists a $\rho \in (0,1)$ and a $C \ge 0$ such that
	\begin{align*}
		\mathbb{E}[|b(t)|^p] = \mathbb{E} \big[ | \Phi_{t,T}(m(\eta_t))(\phi) - \Phi_{t,T}(\mu_t)(\phi) |^p \big] \le C \rho^{p(T-t)}.
	\end{align*}
	Now, for controlling $b(0)$, note that 
	\begin{align*}
		b(0) &= \Phi_{0,T}\big(m(\eta_0)\big)(\phi) - \mu_T(\phi) \\
		&= m(\eta_0)(W_{0,T}(\phi)) - \mu_0(W_{0,T}(\phi)) + \Phi_{0,T}(m(\eta_0))(\phi) - m(\eta_0)(W_{0,T}(\phi))\\
		&= m(\eta_0)(W_{0,T}(\phi)) - \mu_0(W_{0,T}(\phi)) + \Phi_{0,T}(m(\eta_0))(\phi) \big[ 1 - m(\eta_0)(W_{0,T}(\pmb{1})) \big].
	\end{align*}
	Thus, using Assumption \ref{assump:initial_condition} and the fact that $\mu_0 (W_{0,T}(\pmb{1})) = 1$, we get
	\(
	\mathbb{E}[|b(0)|^p] \le {C}/{N^{p/2}}.
	\)
	Let us now establish the global control optimising the choice of the argument $t$ in \eqref{eq:obj_function}.
	We have
	\begin{align}
		\mathbb{E}\left[ |a(0) + b(0)|^p  \right] &\le C \frac{\kappa^{p T} + 1}{N^{p/2}}, \label{ineq:total_control1}\\
		\mathbb{E}\left[ |a(t) + b(t)|^p  \right] &\le C \frac{\kappa^{p (T-t)} + 1}{N^{p/2}} + C \rho^{p(T-t)},\label{ineq:total_control2}
	\end{align}
	for all $t \in [0,T]$.
	
	The key idea now is to find a $t_\epsilon$ such that $\kappa^{t_\epsilon}/N$ and $\rho^{t_\epsilon}$ are both equal to $1/N^\epsilon$. 
	Let us take $t_\epsilon = \frac{\ln N}{\ln \kappa - \ln \rho}$ and $\epsilon = \frac{- \ln \rho}{\ln \kappa - \ln \rho}$.
	Then, we have
	\[
	\frac{\kappa^{ t_\epsilon }}{N} = \exp\left\{- \epsilon \ln N  \right\} = \frac{1}{N^{\epsilon}} \;\; \text{ and } \;\;
	\rho^{ t_\epsilon } = \exp\{ - \epsilon \ln N \} = \frac{1}{N^{\epsilon}}.
	\]
	We thus obtain the desired inequality:
	\[
	\mathbb{E}\left[ | m(\eta_T)(\phi) - \mu_T(\phi) |^p \right] \le \frac{C}{N^{\epsilon p/2}}.
	\]
	Indeed, this inequality is obtained either from \eqref{ineq:total_control1} when $T \le t_\epsilon/2$, since the expression in the upper bound is increasing in $T$, or from \eqref{ineq:total_control2} otherwise taking $T-t = t_\epsilon /2$.
	
\end{proof}

We proceed now to prove the induction step (equation \eqref{eq:induction_step} below), which, together with the initial case proved in Lemma \ref{lemma:keylemma}, concludes the proof of Theorem \ref{thm:uniformLp}. 

\subsection{Proof of Theorem \ref{thm:uniformLp}}
	
	Taking $t = T$ in \eqref{eq:martingale-linear} reduces to
	\begin{align}
		\mathcal{M}_T\big( W_{\cdot, T}(\bar{ \phi}_T) \big) = & \;m(\eta_{T})(\bar{\phi}_T ) - m(\eta_0)(W_{0, T}(\bar{ \phi}_T))  - \int_0^T \big(\mu_s(\Lambda) - m(\eta_s)(\Lambda) \big) m(\eta_{s}) (W_{s, T}(\bar{ \phi}_T)) \mathrm{d}s. \label{eq:martingaleQtT}
	\end{align}
	Hence,
	\begin{align*}
		|m(\eta_T)(\phi) - \mu_T(\phi)|^p &\le C \left( | m(\eta_0)(W_{0,T}(\bar{\phi}_T)) |^p + |\mathcal{M}_T(W_{\cdot, T}(\bar{\phi}_T) )|^p + R_p \right),
	\end{align*}
	where
	\[
	R_p = \left| \int_0^T \big(\mu_s(\Lambda) - m(\eta_s)(\Lambda) \big) m(\eta_{s}) \big( W_{s, T}(\bar{ \phi}_T) \big) \mathrm{d}s \right|^p.
	\]
	
	The initial error can be controlled using Assumption \ref{assump:initial_condition}.
	Indeed,
	\begin{align*}
		\mathbb{E}\left[ | m(\eta_0)(W_{0,T}(\bar{\phi}_T)) |^p \right] &= \mathbb{E}\left[ | m(\eta_0)(W_{0,T}(\phi)) - \mu_T(\phi) + \mu_T(\phi)  - \mu_T(\phi) m(\eta_0) (W_{0,T}(\pmb{1}))|^p \right] \\
		&\le C_1 \mathbb{E}\left[ | m(\eta_0)(W_{0,T}(\phi))  - \mu_0 (W_{0,T}(\phi))|^p \right] + C_1 \mathbb{E}\left[ | \mu_0 (W_{0,T}(\pmb{1}))  - m(\eta_0) (W_{0,T}(\pmb{1}))|^p \right] \\
		&\le \frac{C}{N^{p/2}}.
	\end{align*}
	Furthermore, using Lemma \ref{thm:control_quadratic_variation} and BDG inequality we get
	\(
	\mathbb{E}\left[ |\mathcal{M}_T(W_{\cdot, T})(\bar{\phi}_T)|^p \right] \le {C}/{N^{p/2}}.
	\)
	Note that $\mu_s(W_{s,T}(\bar{\phi}_T)) = 0$.
	Using H\"{o}lder inequality, we obtain
	\begin{align*}
		R_p &= \left| \int_0^T \big(\mu_s(\Lambda) - m(\eta_s)(\Lambda) \big) m(\eta_{s}) \left(\frac{\psi_{s,T} }{\|\psi_{s,T}\|} \right) \|\psi_{s,T}\| \mathrm{d}s \right|^p \\
		&\le \int_0^T \big| \mu_s(\Lambda) - m(\eta_s)(\Lambda) \big|^p \left| m(\eta_{s}) \left(\frac{\psi_{s,T} }{\|\psi_{s,T}\|} \right) \right|^p \|\psi_{s,T}\| \mathrm{d}s \left( \int_0^T \| \psi_{s,T}  \| \mathrm{d}s \right)^{p-1} \\
		&\le \kappa \int_0^T \left| \mu_s\left(\frac{\Lambda}{\|\Lambda\|}\right) - m(\eta_s)\left(\frac{\Lambda}{\|\Lambda\|}\right) \right|^p \left| m(\eta_{s}) \left(\frac{\psi_{s,T} }{\|\psi_{s,T})\|} \right) - \mu_s \left(\frac{\psi_{s,T} }{\|\psi_{s,T}\|} \right)\right|^p \|\psi_{s,T}\| \mathrm{d}s,
	\end{align*}
	where $\psi_{s,T} = W_{s, T}(\bar{ \phi}_T)$.
	Taking expectation and using the Cauchy\,--\,Schwarz inequality yield
	\begin{align}
		\mathbb{E}\left[ \left| \int_0^T \big(\mu_s(\Lambda) - m(\eta_s)(\Lambda) \big) m(\eta_{s}) (W_{s, T}(\bar{ \phi}_T)) \mathrm{d}s \right|^p \right] &\le \kappa \int_0^T I_{2p}(N)  \|W_{s, T}(\bar{ \phi}_T)\| \mathrm{d}s \nonumber \\ 
		&\le K I_{2p}(N). \label{eq:boundI2p}
	\end{align}
	Thus, for every $p \ge 1$ we get the inequality 
	\begin{equation}\label{eq:induction_step}
		I_p(N) \le C \left( \frac{1}{N^{p/2}} + I_{2p}(N) \right),
	\end{equation}
	which using Lemma \ref{lemma:keylemma} reduces to
	\(
	I_p(N) \le {C}/{N^{\min\{2 \epsilon, 1\}p/2}}.
	\)
	By induction we obtain the bound
	\(
	I_p(N) \le {C}/{N^{p/2}}.
	\) \qedsymbol

For a fixed $N$, if the process $(\eta_t)_{t \ge 0}$ generated by $\mathcal{Q}$ allows a stationary distribution $\nu_N$,
then under the hypothesis in Theorem \ref{thm:uniformLp} we get
\begin{equation}\label{eq:conseq_of_main_thm_st_distrib}
	\sup_{ \phi \in \mathcal{B}_1(E) } \mathbb{E}_{\nu_N} \big[ | m(\eta_\infty)(\phi) - \mu_\infty(\phi) |^p \big]^{1/p} \le \frac{C_p}{\sqrt{N}},
\end{equation}
for all $p \ge 1$.

We recall that $\xi_t^{(i)}$ stands for the type of the $i$-th individual, for $1 \le i \le N$, at time $t \ge 0$.
Let us denote by $\mathrm{Law}(\xi_t^{(i)})$ the law of $\xi_t^{(i)}$.

\begin{theorem}[Bias estimate and ergodicity of one particle]\label{thm:TVbound}
	Under Assumptions \ref{assump:initial_condition}, \ref{assump:additive+symmetric} and \ref{assump:ergod_normalised}, there exists a constant $C > 0$ such that
	\[
	\sup_{t \ge 0} \left\| \bar{m}(\eta_t) - \mu_t \right\|_{\mathrm{TV}} \le \frac{C}{N}.
	\]
	Moreover, if the initial distribution of the $N$ particles is exchangeable, then
	\[
	\sup_{t \ge 0} \left\| \mathrm{Law}(\xi_t^{(i)}) - \mu_t \right\|_{\mathrm{TV}} \le \frac{C}{N}.
	\]
\end{theorem}

\begin{proof}[Proof of Theorem \ref{thm:TVbound}]
	
	Taking expectation in \eqref{eq:martingaleQtT} we get
	\[
	\mathbb{E}\left[ m(\eta_T)(\phi) \right] - \mu_T(\phi) = \int_0^T \mathbb{E} \left[ \big(\mu_s(\Lambda) - m(\eta_s)(\Lambda) \big) m(\eta_{s}) \Big(W_{s, T}(\bar{ \phi}_T)\Big) \right] \mathrm{d}s.
	\]
	Using Cauchy\,--\,Schwarz inequality we obtain
	\begin{align*}
		| \mathbb{E}\left[ m(\eta_T)(\phi) \right] - \mu_T(\phi) | &\le C_1 \int_0^T I_2(N) \|W_{s,T}(\bar{\phi}_T)\| \mathrm{d}s \le \frac{C}{N}.
	\end{align*}
	
	Now, assume that initially the $N$ particles are sampled according to an exchangeable distribution.
	Note that
	\[
	\mathbb{E}\left[ \frac{\eta_t(x)}{N} \right] =
	\frac{1}{N} \sum_{i = 1}^N \mathbb{P}[\xi_t^{i} = x] = \mathbb{P}[\xi_t^{j} = x], \;\;\; \forall j \in \{1,\dots, N\},
	\]
	where $\xi_t^{i}$ denotes the position of the $i$-th particle of the process $(\eta_t)_{t \ge 0}$ at time $t \ge 0$.
	Note that the second equality holds because of the exchangeability of the particles.
	Thus, as a consequence of the first part of the theorem and the previous equality we get
	\(
	\|\operatorname{Law}(\xi_t^{(i)}) - \mu_t\|_{\mathrm{TV}} \le {C}/{N}.
	\)
\end{proof}

\end{document}